\documentclass[oneside,11pt]{amsart}
\usepackage{amsmath}
\usepackage{amsthm}

\usepackage[T1]{fontenc}

\usepackage[small]{caption}

\usepackage{amsfonts}
\usepackage{amssymb}
\usepackage{amsxtra}

\newtheorem{thm}{Theorem}[section]

\newtheorem{prop}[thm]{Proposition}
\newtheorem{lem}[thm]{Lemma}
\newtheorem{cor}[thm]{Corollary}

\theoremstyle{definition}
\newtheorem{defn}[thm]{Definition}
\newtheorem{rem}[thm]{Remark}

\newtheorem{exmp}[thm]{Example}

\newtheorem*{construction}{Construction}

\renewcommand{\bar}[1]{\overline{#1}}
\newcommand{\boundary}{\partial}

\newcommand{\set}[2]{\{\,{#1} \mid {#2} \,\}}
\newcommand{\bigset}[2]{ \bigl\{ \, {#1} \bigm| {#2} \, \bigr\} }
\renewcommand{\emptyset}{\varnothing}

\newcommand{\tits}{\boundary_{\text{T}}}

\newcommand{\field}[1]{\mathbb{#1}}

\newcommand{\E}{\field{E}}
\newcommand{\Q}{\field{Q}}

\newcommand{\F}{\mathcal{F}}
\renewcommand{\P}{\field{P}}

\newcommand{\of}{\circ}
\renewcommand{\hat}{\widehat}

\DeclareMathOperator{\CAT}{CAT}

\DeclareMathOperator{\Stab}{Stab}

\DeclareMathOperator{\diam}{diam}

\newcommand{\nbd}[2]{\mathcal{N}_{#2}({#1})}  
\newcommand{\bignbd}[2]{\mathcal{N}_{#2} \bigl( {#1} \bigr)}

\usepackage{combelow}

\newcommand{\Swiatkowski}{{\'{S}}wi{\k{a}}tkowski}

\usepackage{color}
\usepackage{pinlabel}

\usepackage{ifthen}

\newcommand{\showcomments}{yes}

\newsavebox{\commentbox}
{\ifthenelse{\equal{\showcomments}{yes}}%
{\footnotemark
    \begin{lrbox}{\commentbox}
    \begin{minipage}[t]{1.25in}\raggedright\sffamily\upshape\tiny
    \footnotemark[\arabic{footnote}]}
{\begin{lrbox}{\commentbox}}}%
{\ifthenelse{\equal{\showcomments}{yes}}%
{\end{minipage}\end{lrbox}\marginpar{\usebox{\commentbox}}}
{\end{lrbox}}}

\begin{document}

\title[Connectedness properties and isolated flats]{%
Connectedness properties and
splittings of groups with isolated flats}

\author[G.C.~Hruska]{G.~Christopher Hruska}
\address{Department of Mathematical Sciences\\
         University of Wisconsin--Milwaukee\\
         PO Box 413\\
         Milwaukee, WI 53211\\
	 USA}
\email{chruska@uwm.edu}

\author[K.~Ruane]{Kim Ruane}
\address{Department of Mathematics\\
         Tufts University\\
         Medford, MA 02155\\
	 USA}
\email{kim.ruane@tufts.edu}

\begin{abstract}
In this paper we study $\CAT(0)$ groups and their splittings as graphs of groups.
For one-ended $\CAT(0)$ groups with isolated flats we prove a 
theorem characterizing exactly
when the visual boundary is locally connected. This characterization depends on whether the group has a certain type of splitting over a virtually abelian subgroup.
In the locally connected case, we describe the boundary as a tree of metric spaces in the sense of \Swiatkowski.

A significant tool used in the proofs of the above results is a general convex splitting theorem for arbitrary $\CAT(0)$ groups.
If a $\CAT(0)$ group splits as a graph of groups with convex edge groups, then the vertex groups are also $\CAT(0)$ groups.
\end{abstract}



\date{\today}

\maketitle

\section{Introduction}
\label{sec:Introduction}

A major theme of geometric group theory over the last few decades has been the study of a group using various boundaries at infinity attached to spaces on which the group acts.  There has been a fruitful connection between topological properties of these boundaries and algebraic properties of the group.  Our main theorem provides another example of such a connection for the groups that act geometrically on $\CAT(0)$ spaces with isolated flats.

This article is concerned with the following problem:
describe the topological spaces that can arise as the boundary of a $\CAT(0)$ space with isolated flats that admits a geometric group action.  Kapovich--Kleiner posed an analogous question for word hyperbolic groups in \cite{KapovichKleiner00}.
In order to describe boundaries, one first needs to know which topological properties boundaries can have.

If one can prove that the boundary is connected and locally connected, then it is a Peano continuum. This is a class of compact metric spaces with a rich structure theory that is substantially more powerful than general continuum theory.  Much of this paper focuses on determining which groups can have such a boundary.

In the hyperbolic setting, all connected boundaries are locally connected, by a deep, nontrivial theorem (see \cite{BestvinaMess91,Levitt98,Bowditch99Treelike,Swarup96}) with far reaching consequences such as \cite{Bowditch98CutPoints,KapovichKleiner00}.
Groups acting on $\CAT(0)$ spaces with isolated flats \cite{KapovichLeeb95,HruskaKleinerIsolated}  are, in a sense, the simplest generalization of hyperbolicity in the $\CAT(0)$ setting.  By Hruska--Kleiner, a $\CAT(0)$ group with isolated flats is relatively hyperbolic with respect to the collection $\mathbb{P}$ of maximal virtually abelian subgroups of higher rank \cite{HruskaKleinerIsolated}.  A virtually abelian group has \emph{higher rank} if its rank over $\mathbb{Q}$ is at least two.

By a theorem of Mihalik--Ruane \cite{MihalikRuane99,MihalikRuane01}, many one-ended $\CAT(0)$ groups with isolated flats have connected but non--locally connected boundary.  These examples occur when the group exhibits a particular kind of splitting as a graph of groups. 
A simple example of the non--locally connected type is the fundamental group of the space obtained from a closed genus two surface and a torus by gluing together an essential simple closed curve from each surface (see Example~\ref{exmp:SurfaceAmalgam}).

In the general $\CAT(0)$ setting, determining local connectedness seems to be quite a delicate issue.  One-ended hyperbolic groups and fundamental groups of closed Hadamard manifolds all have locally connected boundary, while all $\CAT(0)$ boundaries of one-ended nonabelian right-angled Artin groups are non--locally connected.
Very few natural classes of $\CAT(0)$ groups are known to admit boundaries of both types, and among those none had been completely characterized in terms of local connectivity.
Indeed this question is not even completely understood for right-angled Coxeter groups, a family that includes many groups with both types of boundary.

Let $G$ act geometrically on a one-ended $\CAT(0)$ space $X$
with isolated flats.
Our main theorem determines exactly when $G$ has locally connected boundary.

\begin{thm}[Locally connected]
\label{thm:MainThm}
Let $G$ be a one-ended $\CAT(0)$ group with isolated flats.
The boundary $\boundary G$ is non--locally connected if and only if 
$G$ contains a pair of virtually abelian subgroups $B < A$ with the following properties:
\begin{enumerate}
\item $G$ splits \textup{(}nontrivially\textup{)} over $B$,
\item $A$ has higher rank, and
\item The $\Q$--rank of $B$ is strictly less than the $\Q$--rank of $A$.
\end{enumerate}
\end{thm}

The reverse implication of Theorem~\ref{thm:MainThm} follows immediately from the Mihalik--Ruane splitting theorem.  This article concerns the forward implication of Theorem~\ref{thm:MainThm}, which was not known previously.

A compactum is \emph{semistable} if it is shape-theoretically equivalent to a locally connected continuum \cite{GeogheganSwenson_Semistable}.
A well-known conjecture of Geoghegan--Mihalik would imply that every boundary of every one-ended $\CAT(0)$ group is semistable \cite{Mihalik83,GeogheganSwenson_Semistable}.
This conjecture has been proven in the $\CAT(0)$ with isolated flats setting by combining work of Mihalik--Swenson and the authors \cite{MihalikSwenson,HruskaRuaneHierarchies}.
In other words, all torsion-free one-ended $\CAT(0)$ groups with isolated flats have semistable boundary.
Thus the conclusion of Theorem~\ref{thm:MainThm} gives finer information than one could derive from semistability alone.
When considered up to shape equivalence, all isolated flats boundaries are equivalent to Peano continua, but when considered up to homeomorphism it turns out that many are not Peano continua.
In a sense, semistability is too weak to see the precise homeomorphism type of a given boundary.

The conclusion of Theorem~\ref{thm:MainThm} has already been used as an essential ingredient in constructing examples of non-hyperbolic $\CAT(0)$ groups with boundary homeomorphic to the Menger curve \cite{HaulmarkCAT0,HaulmarkHruskaSathaye_Menger,DaniHaulmarkWalsh_Nonplanar}, which were not previously known to exist.
Although groups with Menger boundary in the hyperbolic setting are well-known \cite{Benakli92_thesis,Champetier95,Bourdon97,KapovichKleiner00,DahmaniGuirardelPrzytycki11}, substantially different techniques are needed in the non-hyperbolic setting.

In the locally connected case we also obtain a detailed description of the boundary of $G$ as the tree of metric compacta in the sense of \Swiatkowski\ \cite{SwiatkowskiTreesOfCompacta}.  This structure explicitly determines the boundary as a tree of spaces built in the same fractal manner as the trees of manifolds of Jakobsche and Ancel--Siebenmann \cite{Jakobsche80,AncelSiebenmann85,Jakobsche91}, which generalize the classical Pontryagin mod-$2$ surface \cite{Pontryagin30}.

The building blocks in this construction are boundaries of vertex groups in a natural graph of groups splitting of $G$.  These vertex groups are atomic in the following sense.
Suppose $G$ is a one-ended $\CAT(0)$ group with isolated flats.
Let $\mathcal{A}$ be the family of subgroups of $G$ that are contained in higher rank virtually abelian subgroups.
We say that $G$ is \emph{atomic} if $G$ does not split over any subgroup in $\mathcal{A}$.
The methods used in the proof of Theorem~\ref{thm:MainThm} also give the following theorem explicitly describing the topology of the boundary.

\begin{thm}[Tree of spaces]
\label{thm:TreeOfSpaces}
Let $G$ be a one-ended $\CAT(0)$ group $G$ with isolated flats.  If $\boundary G$ is locally connected, then $G$ is the fundamental group of a graph of groups such that all vertex groups are atomic $\CAT(0)$ groups with isolated flats, and all edge groups are higher rank virtually abelian.  Furthermore $\boundary G$ is homeomorphic to a tree of metric compacta where the compacta are the boundaries of the vertex groups of the splitting.
\end{thm}

As suggested by \Swiatkowski, the structure of a tree of metric compacta may prove useful in classifying which topological spaces arise as visual boundaries of $\CAT(0)$ groups.

In a general setting, attempts to understand a group by splitting into indecomposable pieces often require a study of hierarchies of splittings (as in the settings of one-relator groups, $3$--manifold groups, cubulated groups, etc.).
For example in \cite{HruskaRuaneHierarchies}, the authors prove semistability in the isolated flats setting using a nontrivial theorem of Louder--Touikan \cite{LouderTouikan17} on the termination of slender hierarchies of relatively hyperbolic groups.  The study of hierarchies is technically much more elaborate than the study of a single splitting.

In contrast with \cite{HruskaRuaneHierarchies}, in the locally connected case of Theorem~\ref{thm:TreeOfSpaces} it is notable that we reach atomic vertex groups after only a single splitting.
We show that the hierarchy terminates after one splitting in Theorem~\ref{thm:BowditchComponent}.
In Corollary~\ref{cor:ComponentCAT0} we show that the resulting vertex groups are $\CAT(0)$ with isolated flats.

\subsection{Methods of proof}
\label{subsec:Methods}

As mentioned above, a group $G$ acting properly, cocompactly on a $\CAT(0)$ space with isolated flats is hyperbolic relative to the family $\mathbb{P}$ of all maximal virtually abelian subgroups of higher rank; ie, rank at least two.
Bowditch introduced a boundary associated to the pair $(G,\mathbb{P})$, now known as the \emph{Bowditch boundary} $\partial (G,\mathbb{P})$.

In the isolated flats setting, the Bowditch boundary is always locally connected \cite{Bowditch01}.
Theorem~\ref{thm:MainThm} does not follow directly from Bowditch's theorem for the following reason.
Hung Cong Tran has shown that the Bowditch boundary is a quotient space of the $\CAT(0)$ boundary \cite{Tran13}.
It is well-known that the continuous image of a locally connected space is also locally connected.  However to prove Theorem~\ref{thm:MainThm} we would need the converse, which is simply not true in general.

The proof in \cite{Bowditch01} that Bowditch's boundary is locally connected involves three main steps.  The first step gives the existence of a maximal peripheral splitting of $G$ (see Section~\ref{sec:PeripheralSplittings}). The vertex groups of this peripheral splitting are relatively hyperbolic groups that admit only a trivial peripheral splitting.  The second step involves showing that groups for which this splitting is trivial have locally connected Bowditch boundary. In the third step Bowditch shows that the boundary is composed of many copies of boundaries of the vertex groups glued along points in the pattern of the Bass--Serre tree.

Our proof of the forward direction of Theorem~\ref{thm:MainThm} is via a contrapositive argument.  We show that if $G$ does not have an ``infinite index'' splitting as in the statement of the theorem, then the $\CAT(0)$ boundary of $G$ is locally connected.
This proof follows a similar outline as in Bowditch's proof.  But the details of each step are substantially different from Bowditch's methods.

The first step begins by examining Bowditch's maximal peripheral splitting.  We can use this splitting result since our groups are relatively hyperbolic.  
However we need the much stronger conclusions that the vertex groups of Bowditch's splitting are atomic
and again $\CAT(0)$ with isolated flats.
The algebraic assumption of no infinite index splitting gives atomic vertex groups (see Theorem~\ref{thm:BowditchComponent}).

In order to conclude that the vertex groups are $\CAT(0)$ groups, we introduce the following ``Convex Splitting Theorem.''
We note that this theorem does not involve the notion of isolated flats, and therefore could prove useful in other $\CAT(0)$ situations.
Suppose $G$ acts geometrically on any $\CAT(0)$ space $X$. A subgroup $H \le G$ is \emph{convex} if $H$ stabilizes a closed convex subspace $Y$ of $X$, and $H$ acts cocompactly on $Y$.

\begin{thm}[Convex Splitting Theorem]
\label{thm:ConvexVertexGroup}
Let $G$ act geometrically on any $\CAT(0)$ space $X$.
Suppose $G$ splits as the fundamental group of a graph of groups $\mathcal{G}$ such that each edge group of $\mathcal{G}$ is convex.
Then each vertex group is also convex.
In particular, each vertex group is a $\CAT(0)$ group itself.
\end{thm}

To see that this theorem is not obvious, recall first that a similar result for hyperbolic groups and quasiconvex subgroups is well-known and very straightforward to prove (see, for example \cite[Prop.~1.2]{Bowditch98CutPoints}).
The basic strategy in the hyperbolic case is to map the Cayley graph of $G$ to the Bass--Serre tree $T$ and to use this map to cut the Cayley graph into pieces that one proves are quasiconvex.

However in the $\CAT(0)$ setting, it is not obvious how to map $X$ to the tree.  To do this, we rely on a powerful theorem of Ontaneda that produces an equivariant simplicial nerve for a $\CAT(0)$ space $X$ with a cocompact group action \cite{Ontaneda05}.
We map $X$ to the tree $T$ using this nerve, and use this map to cut $X$ into pieces (typically not convex).
To complete the proof of Theorem~\ref{thm:ConvexVertexGroup}, much more care is needed than in the hyperbolic case since convexity is more delicate to establish than quasiconvexity. 

The second step of the proof of Theorem~\ref{thm:MainThm} is to prove the following special case involving atomic $\CAT(0)$ groups with isolated flats.  

\begin{thm}
\label{thm:UnsplittableLC}
Let $X$ be a $\CAT(0)$ space with isolated flats that admits a geometric group action by a group $G$.  Suppose $G$ is atomic.
Then the $\CAT(0)$ boundary $\boundary X$ is locally connected.
\end{thm}

The main technical difficulty in the proof of Theorem~\ref{thm:UnsplittableLC} is proving local connectedness at a point that lies in the limit set of a flat Euclidean subspace.  The examination of such points requires the development of new techniques not used in the study of the Bowditch boundary (see Section~\ref{sec:FlatLocalCon}).

The third and final step in the proof of Theorem~\ref{thm:MainThm} is to use the hypothesis of ``no infinite index splittings'' to conclude that the $\CAT(0)$ boundary of $G$ is locally connected.
Most of the work in this step is proving that $G$ has the structure of a tree of metric compacta where the compacta are boundaries of atomic $\CAT(0)$ groups with isolated flats.

The precise topology on this compactification is given as an inverse limit.  Using this description of the boundary of $G$ as an inverse limit, we deduce local connectivity by applying a theorem of Capel on inverse limits of locally connected spaces \cite{Capel54}.

\subsection{Organization of the paper}

Section~\ref{sec:Examples} contains an informal discussion of several examples of $\CAT(0)$ groups with isolated flats, some with locally connected boundary and some with non--locally connected boundary.
Sections \ref{sec:CAT0Boundary}, \ref{sec:RelHyp}, and \ref{sec:IsolatedFlats} summarize necessary background on $\CAT(0)$ spaces, relative hyperbolicity, and groups with isolated flats.
In Section~\ref{sec:PeripheralSplittings}, we examine Bowditch's work on peripheral splittings of relatively hyperbolic groups and deduce several consequences in the setting of isolated flats.
Theorem~\ref{thm:ConvexVertexGroup} is proved in Section~\ref{sec:ConvexSplittings}.
In Section~\ref{sec:RankOneLocalCon} we show that the visual boundary is locally connected at each point that is not in the boundary of a flat, which is the first part of the proof of Theorem~\ref{thm:UnsplittableLC}.
The remaining part of this theorem is proved in Section~\ref{sec:FlatLocalCon}.
Section~\ref{sec:TreeLimits} summarizes definitions, terminology, and key facts about tree systems of metric compacta.
Finally in Section~\ref{sec:PuttingTogether} we show that the boundary has the structure of a tree system as described in Theorem~\ref{thm:TreeOfSpaces}, and we complete the proof of Theorem~\ref{thm:MainThm}.

\subsection{Acknowledgements}
During their work on this project, the authors benefited from many conversations about this work with Ric Ancel, Mladen Bestvina, Craig Guilbault, Matthew Haulmark, Mike Mihalik, Boris Okun, Eric Swenson, Hung Cong Tran, and Genevieve Walsh.  We are grateful for the advice and feedback received during these conversations.

This work was partially supported by a grant from the Simons Foundation (\#318815 to G. Christopher Hruska).

\section{Examples with and without locally connected boundary}
\label{sec:Examples}

In this section we illustrate both directions of Theorem~\ref{thm:MainThm} with examples.
The first example shows a group whose visual boundary is not locally connected, but whose Bowditch boundary is locally connected.
The remaining examples illustrate various constructions of non-hyperbolic groups with locally connected visual boundaries, some illustrating the ``indecomposable'' case of Theorem~\ref{thm:UnsplittableLC} and others constructed using ``locally finite'' amalgams of indecomposable groups.

\begin{exmp}[A non--locally connected boundary]
\label{exmp:SurfaceAmalgam}
Consider the following amalgam of surface groups whose visual boundary is not locally connected.
Let $\Sigma$ be a closed
hyperbolic surface, and $T^2$ be a $2$--dimensional torus
with a fixed Euclidean metric.
Fix a simple closed geodesic loop $\gamma$ in~$\Sigma$.
Choose a closed geodesic $\gamma' \subset T$ such that
$\gamma$ and $\gamma'$ have equal lengths.
Let $X$ be the result of gluing $\Sigma$ to $T^2$ along $\gamma=\gamma'$.
Then the universal cover $\tilde{X}$ of~$X$ is a $\CAT(0)$ space with isolated flats.
The group $G = \pi_1(X)$ is clearly an amalgam of the subgroups $A=\pi_1(T^2)$ and $C= \pi_1(\Sigma)$ amalgamated over the subgroup $B=\langle\gamma\rangle$.
Since this splitting $A *_B C$ satisfies the conditions of Theorem~\ref{thm:MainThm},
we see that the visual boundary $\boundary \tilde{X}$ is not locally connected.
\end{exmp}

In order to illustrate the difference between the non--locally connected $\CAT(0)$ boundary of $G$ and the locally connected Bowditch boundary $\boundary(G,\mathbb{P})$ in Example~\ref{exmp:SurfaceAmalgam}, we briefly sketch a significant non--locally connected subset of the boundary.  We then describe its locally connected image in the Bowditch boundary.

Each flat $F$ in $\tilde{X}$ is a lift of the torus $T^2$.  Inside $F$ are infinitely many lifts of the geodesic $\gamma$, along each of which there is a copy of $\mathbb H^2$ attached.  The lifts of $\gamma$ in the flat $F$ are all parallel and thus all share the same pair of endpoints $a,b$ in $\partial F$.
Let $Y$ be the convex subcomplex of $\tilde{X}$ consisting of the flat~$F$ along with the countably many hyperbolic planes glued to~$F$ along the lifts of $\gamma$. Then $\boundary Y$ is a suspension of a set $K$ that is countably infinite with two limit points. The suspension points are $a$ and $b$.  Note that this suspension itself is not locally connected.  Furthermore, each of the points in $\boundary F - \{a,b\}$ turns out to be a point of non--local connectivity in $\boundary \tilde{X}$.

To understand what the Bowditch boundary is in this example, note that there are two types of circles in the visual boundary $\boundary\tilde X$---those that occur as the boundary of a hyperbolic plane and those that occur as the boundary of a flat.  By \cite{Tran13}, the Bowditch boundary is obtained from the visual boundary by collapsing the circles arising as boundaries of flats.  In the quotient, each such circle becomes a global cut point of the Bowditch boundary.  Each of these cut points is incident to a countable family of circles whose union forms a Hawaiian earring.

\begin{exmp}[Some locally connected boundaries]
In this example we show three different groups with locally connected boundary, formed by gluing hyperbolic $3$--manifolds along cusps.

First consider the figure eight knot $K \subset S^3$, and let $N$ be a closed regular neighborhood of $K$.  Let $M^3$ be the compact knot complement, $S^3$ minus the interior of $N$.
It is well-known that $G = \pi_1(M^3)$ is one-ended with isolated flats and does not split over any subgroup of the cusp group. Thus by Theorem~\ref{thm:UnsplittableLC} the visual boundary of the $\CAT(0)$ space $\tilde{M^3}$ is locally connected.
In this simple example, it was known previously that the boundary was locally connected, since it is homeomorphic to a Sierpinski carpet by \cite{Ruane05Sierpinski}.
(In this case the Bowditch boundary is a $2$--sphere formed by collapsing each peripheral circle of the Sierpinski carpet to a point.)

If we double $M^3$ along its boundary torus $\boundary N$, we get a closed $3$--manifold consisting of two hyperbolic pieces glued along the torus $T^2 = \boundary N$.
The fundamental group $D$ of the double does not split over any cyclic subgroup of $\pi_1(T^2)$, and thus by Theorem~\ref{thm:MainThm} its boundary is also locally connected.
Once again, we knew this already because the visual boundary of any closed nonpositively curved $3$--manifold group is a $2$--sphere.
Applying the proof of Theorem~\ref{thm:MainThm} to this example recovers the classical decomposition of $S^2$ as a tree of Sierpinski carpets glued in pairs along peripheral circles.

If we form a ``triple'' of $M^3$ instead of a double, we get a less familiar example, whose visual boundary was not previously known to be locally connected.
As with the double, the tripled space is formed by gluing three copies of $M^3$ along the boundary torus $T^2$.  As above its fundamental group does not split over any cyclic subgroup of $\pi_1(T^2)$.
Thus we establish that its visual boundary is locally connected.
We also obtain a description of this boundary as a $2$--dimensional compactum formed as the limit of a tree of Sierpinski carpets, this time glued in triples along peripheral circles.
\end{exmp}

\begin{exmp}
We conclude this section with an example of a group with isolated flats having Serre's Property FA---i.e., no splittings at all.
Consider any Coxeter group $W$ on $5$ generators $s_i$ of order two with defining Coxeter relations $(s_i s_j)^{m_{ij}} = 1$ such that $3\le m_{ij} < \infty$ for all $i\ne j$.
The Davis complex $\Sigma$ of each such $W$ is a piecewise Euclidean $\CAT(0)$ $2$--complex whose $2$--cells are isometric to regular Euclidean $2m_{ij}$--gons with at least $6$ sides.
By an observation of Wise, such complexes have isolated flats
(see \cite{Hruska2ComplexIFP}).
Since each $m_{ij}$ is finite, $W$ has Property FA by \cite{Serre77}.
Therefore $W$ is one-ended with locally connected visual boundary by Theorem~\ref{thm:UnsplittableLC}.
Haulmark--Hruska--Sathaye show in \cite{HaulmarkHruskaSathaye_Menger} that each such $W$ has $\boundary W$ homeomorphic to the Menger curve, using the local connectivity established above as a key step.
\end{exmp}

\section{The visual boundary of a $\CAT(0)$ space}
\label{sec:CAT0Boundary}

We refer the reader to \cite{Ballmann95,BH99} for introductions
to the theory of $\CAT(0)$ spaces.
Throughout this section $X$ is assumed to be a proper $\CAT(0)$ space, a condition that holds whenever $X$ admits a proper, cocompact, isometric group action.

The $\CAT(0)$ geometry $X$ gives rise to the visual boundary $\boundary X$, which is a compact metrizable space.
We first, define the boundary $\boundary X$ as a set as follows:

\begin{defn}[Visual boundary as a set]
Two geodesic rays $c,c^\prime\colon [0,\infty)\to
X$ are said to be \emph{asymptotic} if there exists a constant $K$ such
that $d\bigl(c(t),c^\prime (t)\bigr)\leq K$ for all $t>0$---this is an equivalence
relation.  The boundary of $X$, denoted $\partial X$, is then the set of
equivalence classes of geodesic rays. The equivalence class of a ray $c$ is
denoted by $c(\infty)$.

Since $X$ is complete, then for each basepoint $q\in X$
and each $\xi \in \boundary X$
there is a unique geodesic $c$ such that $c(0)=q$ and $c(\infty)=\xi$.
Thus we may identify $\boundary X$ with the set $\boundary_q X$ of all rays
emanating from $q$.
We use the notation $\bar{X} = X \cup \boundary X$.
\end{defn}

\begin{defn}[The cone topology on $\bar{X}$]
There is a natural topology on $\bar{X}$ called the cone topology, which is defined in terms of the following neighborhood basis.  Let $c$ be a geodesic segment or ray, let $q = c(0)$, and choose any $r>0$ and $D >0$. Also, let $\overline{B}(q,r)$ denote the closed ball of radius $r$ centered at $q$ with $\pi_r\colon \overline{X}\to \overline{B}(q,r)$ denoting projection.  Define
\[
   U(c,r,D)=\bigset{ x\in\overline{X} }{ d(x,q)>r,\ d \bigl(\pi_r(x),c(r) \bigr) < D }
\]
This consists of all points in $\overline{X}$ such that when projected back to $\overline{B}(q,r)$, this projection is not more than $D$ away from the intersection of the sphere with $c$.  These sets along with the metric balls in $X$ form a basis for the \emph{cone topology} on $\overline X$.  The induced topology on $\partial X$ is also called the cone topology on $\partial X$, and the resulting topological space is the \emph{visual boundary} of $X$.
We occasionally use the notation $U(c,r,D)$ to refer to basic neighborhoods in the visual boundary.
Since $X$ is proper, both $\overline{X}$ and the visual boundary are compact.
\end{defn}

It is a well-known result that for any proper $\CAT(0)$ space, both $X$ and $\bar{X}$ are ANR's (Absolute Neighborhood Retracts).
We refer the reader to \cite{Ontaneda05} for a proof that $X$ is an ANR, using a theorem of \cite{Hu65}.  See also \cite[\S 2.9]{Guilbault14} for a proof that $\bar{X}$ is an ANR using work of \cite{Hanner51}.

The main consequence of being an ANR that we will use is that $\bar{X}$ is locally connected.  This consequence is much easier to prove directly, which we do below.

\begin{prop}
\label{prop:BarXLC}
$\bar{X}$ is locally connected.  Furthermore, each point $\xi \in \boundary X$ has a connected neighborhood $\bar{N}$
such that $N = \bar{N} \cap X$ is a connected set in $X$ and each point of $\Lambda = \bar{N} \cap \boundary X$ is a limit point of $N$.
\end{prop}

Let us pause for a moment to warn the reader that the notation $\bar{N}$ does not refer to a closed set of $\bar{X}$, but rather refers to the fact that each point of $\bar{N}$ is a limit point of $N$.

\begin{proof}
Metric balls in $X$ are connected.
Furthermore the basic open sets $\bar{N}=U(c,r,D)$ are also connected.
Indeed a point $p\in X$ lies in $N = U(c,r,D) \cap X$ if and only if the geodesic segment $c'= \bigl[c(0),p \bigr]$ intersects $B\bigl( c(r), D \bigr)$.
For any such point $p = c'(t) \in N$, we can choose $s$ so that $d\bigl(c'(s),c(r)\bigr) <D$.  It follows that the entire subpath of $c'$ from $c'(s)$ to $p=c'(t)$ is contained in $N$.
Therefore every point of $N$ lies in a connected subset of $N$ that intersects the connected ball 
$B\bigl( c(r), D \bigr)$.  In other words, $N$ is connected.
The entire set $\bar{N} = U(c,r,D)$ is contained in the closure of $N$ within the space $\bar{X}$, so $\bar{N}$ is also connected.
\end{proof}

Throughout this paper we will often need to compare the sizes of different subsets of the visual boundary. To make such a comparison, it is useful to consider explicit metrics on the visual boundary.  The following definition due to Osajda introduces a family of such metrics.

\begin{defn}[Boundary as a metric space]
\label{defn:BoundaryMetric}
Fix a basepoint $q \in X$. For each $D >0$ and $r<\infty$ we define a metric $d_D$ on $\boundary_q X$ as follows.
Given distinct rays $c$ and $c'$ based at $q$, the distance function $t \mapsto d\bigl( c(t),c'(t) \bigr)$ monotonically increases from $0$ to $\infty$.  Thus there exists a unique $r \in (0,\infty)$ such that $d\bigl( c(r),c'(r) \bigr) = D$.  We set $d_D(c,c') = 1/r$.  The function $d_D$ is a metric compatible with the cone topology by Osajda--\Swiatkowski\ \cite{OsajdaSwiatkowski15}. (See also \cite{Moran16}.)
\end{defn}

\section{Relatively hyperbolic groups and their boundaries}
\label{sec:RelHyp}

In this section we define the notions of relative hyperbolicity and the Bowditch boundary.  The definitions we use are due to Yaman \cite{Yaman04} and are given in terms of dynamical properties of an action on a compact space, which turns out to be the Bowditch boundary.

A \emph{convergence group action} is an action
of a finitely generated group $G$ on a compact, metrizable space $M$ satisfying the following conditions, depending on the cardinality
of $M$:
\begin{itemize}
\item If $M$ is the empty set, then $G$ is finite.
\item If $M$ has exactly one point, then $G$ is infinite.
\item If $M$ has exactly two points, then $G$ is virtually cyclic.
\item If $M$ has at least three points, then the action of $G$ on the space
of distinct (unordered) triples of points of $M$ is properly discontinuous.
\end{itemize}
In the first three cases the convergence group action is \emph{elementary}, and in the final
case the action is \emph{nonelementary}.

Suppose $G$ has a convergence group action on $M$.
An element $g \in G$ is \emph{loxodromic} if it has infinite order
and fixes exactly two points of $M$.
A subgroup $P \le G$ is a \emph{parabolic subgroup} if it is infinite and
contains no loxodromic element.
A parabolic subgroup $P$ has a unique fixed point in $M$, called a
\emph{parabolic point}.
The stabilizer of a parabolic point is always a maximal parabolic group.
A parabolic point $p$ with stabilizer $P := \Stab_G(p)$ is \emph{bounded parabolic}
if $P$ acts cocompactly on $M - \{p\}$.
A point $\xi \in M$ is a \emph{conical limit point} if
there exists a sequence $(g_i)$ in $G$ and distinct points
$\zeta_0,\zeta_1 \in M$
such that $g_i(\xi) \to \zeta_0$, while for all $\eta \in M - \{\xi\}$
we have $g_i(\eta) \to \zeta_1$.

\begin{defn}[Relatively hyperbolic]
\label{def:RelHyp}
A convergence group action of $G$ on $M$ is \emph{geometrically finite}
if every point of $M$ is either a conical limit point or a
bounded parabolic point.
If $\P$ is a collection of subgroups of $G$, then the pair $(G,\P)$ is \emph{relatively hyperbolic} if $G$ admits a geometrically finite convergence group action on a compact, metrizable space $M$ such that $\P$ is equal to the collection of all maximal parabolic subgroups.
\end{defn}

\begin{defn}[Bowditch boundary]
By \cite{Yaman04} and \cite{BowditchRelHyp} the space $M$ is uniquely determined by $(G,\P)$ in the following sense:
Any two spaces $M$ and $M'$ arising from the previous definition are $G$--equivariantly homeomorphic.
The compactum $M$ is the \emph{Bowditch boundary} of $(G,\P)$, and will be denoted by $\boundary(G,\P)$.
\end{defn}

We remark that by work of Yaman, the Bowditch boundary can be obtained as the Gromov boundary of a certain $\delta$--hyperbolic space on which $G$ acts.
Although we will not use this point of view in the present article, the construction of this $\delta$--hyperbolic space is the basis of Yaman's proof that Definition~\ref{def:RelHyp} is equivalent to other definitions appearing in the literature (such as those in \cite{BowditchRelHyp}).

Suppose $G$ acts properly, cocompactly, and isometrically on a $\CAT(0)$ space $X$.  Suppose also that $G$ has a family of subgroups $\P$ such that $(G,\P)$ is relatively hyperbolic.
In this case, we have introduced two different types of boundary that one may associate with $G$: the visual boundary $\boundary X$ of the $\CAT(0)$ space $X$ and the Bowditch boundary $\boundary (G,\P)$.
These boundaries are closely related by a theorem of Hung Cong Tran, as was mentioned in the introduction.  We give a precise statement here.

\begin{thm}[\cite{Tran13}]
\label{thm:BoundaryQuotient}
Let $(G,\P)$ and $X$ be as above.
The quotient space formed from $\boundary X$ by collapsing the limit set of each $P \in \P$ to a point
is $G$--equivariantly homeomorphic to the Bowditch boundary $\boundary(G,\P)$.
\end{thm}

\section{Isolated flats}
\label{sec:IsolatedFlats}

A \emph{$k$--flat} in a $\CAT(0)$ space $X$ is an isometrically embedded
copy of Euclidean space $\E^k$ for some $k\ge 2$.
In particular, note that a geodesic line is
not considered to be a flat.

\begin{defn}\label{def:IsolatedFlats}
Let $X$ be a $\CAT(0)$ space, $G$ a group acting
geometrically on $X$, and $\F$ a $G$--invariant set of flats in $X$.
We say that $X$ has \emph{isolated flats with respect to $\F$}
if the following two conditons hold.
\begin{enumerate}
  \item \label{item:AllFlats}
  There is a constant~$D$ such that every flat $F \subset X$ lies in a
  $D$--neighborhood of some $F' \in \F$.
  \item \label{item:ControlledInt}
  For each positive $r< \infty$ there is a constant
  $\rho = \rho(r) < \infty$ so that for any two distinct flats
  $F,F' \in \F$ we have
  \[
     \diam \bigl( \nbd{F}{r} \cap \nbd{F'}{r} \bigr) < \rho.
  \]
\end{enumerate}
We say $X$ has \emph{isolated flats} if it has isolated flats with 
respect to some $G$--invariant set of flats.
\end{defn}

\begin{thm}[\cite{HruskaKleinerIsolated}]
\label{thm:PeriodicFlats}
Suppose $X$ has isolated flats with respect to~$\F$.
For each $F \in \F$ the stabilizer $\Stab_G(F)$ is virtually
abelian and acts cocompactly on~$F$.
The set of stabilizers of flats $F \in \F$ is precisely the set of
maximal virtually abelian subgroups of~$G$ of rank at least two.
These stabilizers lie in only finitely many conjugacy classes.
\end{thm}

\begin{thm}[\cite{HruskaKleinerIsolated}]
\label{thm:HKMain}
Let $X$ have isolated flats with respect to~$\F$.  Then the following
properties hold.
\begin{enumerate}
  \item \label{item:RelHyp}
  $G$ is relatively hyperbolic with respect to the collection
  of all maximal virtually abelian subgroups of rank at least two.
  \item \label{item:TitsBoundary}
  The connected components of the Tits boundary $\tits X$
  are isolated points together with the boundary spheres
  $\tits F$ for all $F \in \F$.
  \end{enumerate}
\end{thm}

The previous theorem also has the following converse.

\begin{thm}[\cite{HruskaKleinerIsolated}]
\label{thm:HKConverse}
Let $G$ be a group acting geometrically on a $\CAT(0)$ space $X$.  Suppose $G$ is relatively hyperbolic with respect to a family of virtually abelian subgroups. Then $X$ has isolated flats.
\end{thm}

A group $G$ that admits an action on a $\CAT(0)$ space with isolated flats has a ``well-defined'' visual boundary, often denoted by $\boundary G$, by the following theorem.

\begin{thm}[\cite{HruskaKleinerIsolated}]
\label{thm:WellDefinedBoundary}
Let $G$ act properly, cocompactly, and isometrically on two
$\CAT(0)$ spaces $X$ and $Y$.  If $X$ has isolated flats, then so does $Y$, and there is a $G$--equivariant homeomorphism $\boundary X \to \boundary Y$.
\end{thm}

Finally, we point out a key geometric fact about $\CAT(0)$ spaces with isolated flats that will be used several times throughout this paper.

\begin{thm}\label{thm:facts}
Suppose $X$ is a $\CAT(0)$ space with isolated flats with respect to $\mathcal{F}$.
There exists a constant $\kappa>0$, such that the following holds:
Given a point $x$, a flat $F \in \mathcal{F}$,
with $c\colon [a,b] \to X$ the shortest path from $x$ to $F$, we have $c \cup F$ is
$\kappa$--quasiconvex in $X$.
More precisely, if $c'$ is any geodesic joining a point of $c$ to a point of $F$, then $c'$ intersects $B \bigl( c(b),\kappa \bigr)$.
\end{thm}

\begin{proof}
If the claim were false, there would be sequences
of flats $F_i \in \mathcal{F}$ and points $x_i \in X$
and $q_i, y_i \in F_i$ such that $[x_i,q_i]$ is a shortest path
from $x_i$ to $F_i$
and $d\bigl(q_i,[x_i,y_i]\bigr)$ tends to infinity.

Pass to a subsequence and translate by the action of $G$ so
that $F_i=F$ is constant.
After passing to a further subsequence, the points $q_i$, $x_i$, and $y_i$ converge respectively to $q \in F$, $\xi_x \in \boundary X$, and
$\xi_y \in \boundary F$.
furthermore, $\xi_x \notin \boundary F$ since the ray from $q$ to $\xi_x$
meets $F$ orthogonally.
Since $d\bigl(q_i,[x_i,y_i]\bigr)$ tends to infinity, it follows
from \cite[Corollary~7]{HruskaKleinerErratum}
that $d_T(\xi_x,\xi_y) \le \pi$,
contradicting Theorem~\ref{thm:HKMain}(\ref{item:TitsBoundary}).
\end{proof}

\section{Peripheral splittings}
\label{sec:PeripheralSplittings}

In this section, we give the definition of a peripheral splitting of a relatively hyperbolic group and examine some of their basic properties.
We introduce the notion of a locally finite peripheral splitting, which plays a key role in the proof of Theorem~\ref{thm:MainThm}.
The goal of this section is to prove Theorem~\ref{thm:BowditchComponent}, which roughly states that in the locally finite case many features of a relatively hyperbolic group are inherited by the vertex groups of its maximal peripheral splitting.

\begin{defn}[Peripheral splittings]
\label{def:peripheral}
Suppose $(G,\P)$ is a relatively hyperbolic group.  A \emph{peripheral splitting} of $(G,\P)$ is a splitting of $G$ as a finite bipartite graph of groups $\mathcal G$, whose vertices have two types that we call \emph{peripheral} vertices and \emph{component} vertices.
We require that the collection of subgroups of $G$ conjugate to the peripheral vertex groups is identical to the collection $\P$ of all peripheral subgroups.
We also require that $\mathcal{G}$ does not contain a component vertex of degree one that is contained in the adjacent peripheral group.
\end{defn}

Such a splitting is called \emph{trivial} if one of the vertex groups is equal to $G$ and \emph{nontrivial} otherwise.  One peripheral splitting $\mathcal H$ is a \emph{refinement} of another $\mathcal G$ if $\mathcal G$ can be obtained from $\mathcal H$ by a finite sequence of foldings of edges
that preserve the colors of vertices.
A refinement is \emph{trivial} if it involves no folds, i.e., the refinement is an isomorphism of graphs.
A peripheral splitting of $(G,\P)$ is \emph{maximal} if it admits only trivial refinements.

Let $\mathcal{A}$ be the collection of all subgroups of the peripheral subgroups of $G$.
If $(G,\P)$ admits only a trivial peripheral splitting, then $G$ does not split over $\mathcal{A}$ relative to $\mathbb{P}$ in the following sense: let $G$ act without inversions on a simplicial tree $T$ such that each peripheral subgroup has a fixed point in $T$.  Then the action of $G$ on $T$ has a global fixed point.

The notion of an atomic group, mentioned in the introduction, is a much stronger condition that does not include the phrase ``relative to $\mathbb{P}$.''

\begin{defn}[Atomic]
\label{defn:Atomic}
We say that $(G,\P)$ is \emph{atomic} if $G$ is one-ended, each $P \in \P$ is one-ended, and $G$ cannot be expressed as an HNN extension or a nontrivial amalgam over any subgroup in $\mathcal{A}$.
\end{defn}

\begin{thm}[Bowditch]
\label{thm:Bowditch}
Suppose $(G,\P)$ is relatively
hyperbolic, $G$ is one-ended and each $P \in \P$ is finitely presented, does not contain an infinite torsion group, and is either one-ended or two-ended.
\begin{enumerate}
\item \cite[Prop.~10.1]{BowditchRelHyp} $\boundary(G,\P)$ is connected.
\item \cite[Thm.~1.5]{Bowditch01}
$\boundary(G,\P)$ is locally connected.
\item \cite[Thm.~1.4]{Bowditch01}
$(G,\P)$ has a unique maximal peripheral splitting, which could be trivial.
\item \cite[\S 9]{Bowditch01}
\label{item:BowditchCutPoint}
If the maximal peripheral splitting is trivial, then $\boundary(G,\P)$ does not contain a global cut point.
\end{enumerate}
\end{thm}

If $(G,\mathbb{P})$ is atomic then the last conclusion of this theorem holds.
In general the component vertex groups of the maximal peripheral splitting inherit a natural relatively hyperbolic structure, but they do not need to be atomic. 
See \cite{HruskaRuaneHierarchies} for examples of relatively hyperbolic groups whose component groups are not atomic, and which require a nontrivial hierarchy of splittings over parabolic subgroups in order to reach atomic groups.

In order to ensure that the maximal peripheral splitting has atomic component vertex groups, we need to focus on the special case of locally finite splittings, which are better behaved than the general case.
A peripheral splitting $\mathcal{G}$ is \emph{locally finite} if for each peripheral vertex group $P$ the adjacent edge groups include as finite index subgroups of $P$.
Equivalently, in the Bass--Serre tree $T$ for $\mathcal{G}$, the vertex $v$ stabilized by $P$ has finite valence.
The following theorem summarizes key properties of the component vertex groups in the unique maximal peripheral splitting in the special case that this splitting is locally finite.
For simplicity we have stated this result only in the isolated flats setting, which is the only case needed in this paper.

\begin{thm}
\label{thm:BowditchComponent}
Let $(G,\P)$ be relatively hyperbolic such that $G$ is one-ended and each $P \in \P$ is virtually abelian of $\Q$--rank at least two.
Suppose the maximal peripheral splitting $\mathcal{G}$ of $(G,\P)$ is locally finite.

For each component vertex group $H$ of $\mathcal{G}$, let $\mathbb{O}$ be the collection of infinite groups of the form $H \cap P$ for all $P \in \P$.
Then 
\begin{enumerate}
    \item
    \label{item:ComponentRelHyp}
    $(H,\mathbb{O})$ is relatively hyperbolic.
    \item Each $O \in \mathbb{O}$ is virtually abelian of $\Q$--rank at least two.
    \item $(H,\mathbb{O})$ is atomic.
\end{enumerate}
\end{thm}

\begin{proof}
Let $T$ be the Bass--Serre tree of the splitting $\mathcal{G}$.  Then $H$ is stabilizer of a vertex $v$ in $T$.
Since the graph of $\mathcal{G}$ is bipartite, the other end of each such edge $e$ is a peripheral vertex $w$.
It follows that the family $\mathbb{O}$ is the collection of stabilizers of edges $e$ adjacent to $v$.
By the locally finite hypothesis, each $O \in \mathbb{O}$ is virtually abelian of $\mathbb{Q}$--rank at least two.
By \cite{Bowditch01}, the pair $(H,\mathbb{O})$ is relatively hyperbolic and $\boundary(H,\mathbb{O})$ is connected.

Since each $O \in \mathbb{O}$ is one-ended and the Bowditch boundary is connected, it follows that $H$ is a one-ended group by Proposition~10.1 of \cite{BowditchRelHyp}.

By Lemma~4.6 of \cite{Bowditch01}, the pair $(H,\mathbb{O})$ admits only a trivial peripheral splitting.
In general that lemma does not imply that $H$ has no splittings over subgroups of members of $\mathbb{O}$.
However since the members of $\mathbb{O}$ are one-ended, we may apply \cite[Proposition~5.2]{Bowditch01} to conclude that $(H,\mathbb{O})$ is atomic.
\end{proof}

\section{Convex splittings of $\CAT(0)$ groups}
\label{sec:ConvexSplittings}

In this section we prove Theorem~\ref{thm:ConvexVertexGroup}, which is a convex splitting theorem for $\CAT(0)$ groups. This splitting theorem does not use the notion of isolated flats, and could potentially be a useful tool in many other $\CAT(0)$ settings.

We apply this splitting theorem to peripheral splittings of $\CAT(0)$ groups with isolated flats, in order to prove Theorem~\ref{thm:ComponentCAT0}, which states that the vertex groups of such a splitting are also $\CAT(0)$ groups with isolated flats.  Although this article is mainly focused on locally finite splittings, Theorem~\ref{thm:ComponentCAT0} has no such restriction and applies to all $\CAT(0)$ groups with isolated flats.  Ben-Zvi uses Theorem~\ref{thm:ComponentCAT0} to show that all one-ended $\CAT(0)$ groups with isolated flats have globally path connected boundary \cite{Benzvi_PathCon}.

The next two results are key tools used in the proof of Theorem~\ref{thm:ConvexVertexGroup}.

\begin{lem}
\label{lem:MapToTree}
Suppose $G$ acts geometrically on a $\CAT(0)$ space $X$, and $G$ also acts on a simplicial tree $T$.
Then there is a $G$--equivariant map $\pi\colon X \to T$ that is continuous with respect to the CW--topology on $T$.
\end{lem}

\begin{proof}
The proof depends on the following result of Ontaneda  \cite{Ontaneda05}:
Let $X$ be a $\CAT(0)$ space on which $G$ acts properly, cocompactly, and isometrically. Then there exists a locally finite, finite dimensional simplicial complex $K$ on which $G$ acts properly, cocompactly, and simplicially.  Furthermore there is a $G$--equivariant continuous map $X \to K$.

To complete the proof, we need a $G$--equivariant continuous map $K \to T$.  After replacing $T$ with its barycentric subdivision, we may assume that $G$ acts on $T$ without inversions.
Choose a representative $\sigma$ for each $G$--orbit of $0$--simplices.  Since $G$ acts properly on $K$, the $G$--stabilizer of $\sigma$ is a finite group $K_\sigma$.
Let $v_\sigma$ be a vertex of $T$ fixed by $K_\sigma$.
We define the map $K^{(0)} \to T$ by mapping $\sigma \mapsto v_\sigma$ and extending equivariantly.
Since $T$ is contractible, we may extend this map to the higher skeleta of $K$ in an equivariant fashion.
\end{proof}

The following folk result has been used implicitly in many places throughout the literature (see for example \cite{HruskaKleinerIsolated}).
We have decided to include the (short) proof for the benefit of the reader.
Dani Wise has described the proof as a ``pigeonhole principle'' for cocompact group actions.

\begin{prop}[Pigeonhole]
\label{prop:pigeonhole}
Suppose a group $G$ acts cocompactly and isometrically on a metric space $X$.
Let $\mathcal{A}$ be a family of closed subspaces of $X$.
Suppose $\mathcal{A}$ is $G$--equivariant and locally finite, in the sense that each compact set of $X$ intersects only finitely many members of $\mathcal{A}$.
Then the stabilizer of each $A \in \mathcal{A}$ acts cocompactly on $A$.
Furthermore the members of $\mathcal{A}$ lie in finitely many $G$--orbits.
\end{prop}

\begin{proof}
Let $K$ be a compact set whose $G$--translates cover~$X$.
Since $K$ intersects only finitely members of $\mathcal{A}$, the sets of $\mathcal{A}$ lie in finitely many orbits.
Thus we only need to establish that for each $A \in \mathcal{A}$ the group $H = \Stab_G(A)$ acts cocompactly on $A$.
Let $\{g_i\}$ be a set of group elements such that the
translates $g_i(K)$ cover~$A$ and each $g_i(K)$ intersects $A$.
If the sets $g_i^{-1}(A)$ and $g_j^{-1}(A)$ coincide, then
$g_jg_i^{-1}$ lies in $H$. 
It follows that the $g_i$ lie in only finitely many right cosets $Hg_i$.
In other words, the sets $g_i(K)$ lie in only finitely many $H$--orbits.
But any two $H$--orbits $Hg_i(K)$ and $Hg_j(K)$ lie at a finite Hausdorff
distance from each other.
Thus any $g_i(K)$ can be increased to a larger compact set $K'$ so that the
translates of $K'$ under~$H$ cover~$A$.
Since $A$ is closed, it follows that $H$ acts cocompactly on $A$, as desired.
\end{proof}

The proof of Theorem~\ref{thm:ConvexVertexGroup} will be developed over the course of the next several lemmas and definitions.

Let $\pi\colon X \to T$ be the $G$--equivariant continuous map given by Lemma~\ref{lem:MapToTree}.
For each edge $e$ of $T$, let $m_e$ denote the midpoint of $e$.  For each vertex $v$ of $T$, let $S(v)$ be the union of all segments of the form $[v,m_e]$ where $e$ is an edge adjacent to $v$; in other words, $S(v)$ is the union of all half-edges emanating from $v$.
Let $Q(e) \subset X$ be the preimage $\pi^{-1}(m_e)$, and let $Q(v) \subset X$ be the preimage $\pi^{-1}\bigl(S(v)\bigr)$.

\begin{lem}
\label{lem:CocompactEdge}
Let $G_e$ be the $G$--stabilizer of the edge $e$ of $T$.  Then $G_e$ acts cocompactly on $Q(e)$.
\end{lem}

\begin{proof}
The family $\mathcal{A} = \bigset{Q(e)}{\text{$e$ an edge of $T$}}$ is clearly $G$--equivariant.
Each $Q(e)$ is closed in $X$ since it is the preimage of the closed set $\{m_e\}$ under a continuous map.
We will show that $\mathcal{A}$ is locally finite, and that the $G$--stabilizer of each $Q(e)$ is equal to the group $G_e$.

In order to see local finiteness, let $K$ be a compact set in $X$.  Since $\pi \colon X \to T$ is continuous with respect to the CW topology on $T$, the image $\pi(K)$ intersects at most finitely many open edges of $T$.  In particular $\pi(K)$ contains only finitely many midpoints $m_e$ of edges.  It follows that $K$ intersects only finitely many sets $Q(e)$.

The equivariance of $\pi$ implies that an element $g \in G$ satisfies $g \bigl( Q(e) \bigr) = Q(e)$ if and only if $g(m_e) = m_e$.  Thus $\Stab\bigl( Q(e) \bigr) = \Stab(m_e) = G_e$.  The result now follows from Proposition~\ref{prop:pigeonhole}.
\end{proof}

\begin{lem}
\label{lem:CocompactVertex}
Let $G_v$ be the $G$--stabilizer of the vertex $v$ of $T$.
Then $G_v$ acts cocompactly on $Q(v)$.
\end{lem}

\begin{proof}
We only need to observe that when projecting a compact set $K$ to the tree $T$, the image $\pi(K)$ intersects only finitely many sets $S(v)$ since it intersects only finitely many open edges of $T$. The rest of the proof is identical to the proof of Lemma~\ref{lem:CocompactEdge}.
\end{proof}

In the proof of Theorem~\ref{thm:ConvexVertexGroup} we assume that we have a splitting whose edge groups are convex, in other words each edge group $G_e$ stabilizes a closed convex subspace $Y_e$ on which it acts cocompactly.
Although $G_e$ also acts cocompactly on $Q(e)$, it is unlikely that $Q(e)$ itself is a convex subspace of $X$.
In the next lemma, we show that we can enlarge $Q(e)$ to its convex hull and preserve cocompactness.

\begin{lem}
Suppose the edge group $G_e$ cocompactly stabilizes a convex subspace $Y_e$ of the $\CAT(0)$ space $X$.
Let $C(e)$ be the closure of the convex hull of $Q(e)$.
Then $G_e$ acts cocompactly on $C(e)$.
\end{lem}

\begin{proof}
Let $K$ be a compact set of $X$ whose $G_e$--translates cover $Q(e)$.  Then $K$ is contained in the closed neighborhood $\overline{\nbd{Y_e}{D}}$ for some $D$. 
Since $X$ is proper, any compact set $K'$ whose translates cover $Y_e$ can be increased to a larger compact set $\overline{\nbd{K'}{D}}$ whose $G_e$--translates cover the closed neighborhood $\overline{\nbd{Y_e}{D}}$.
This neighborhood is a closed, convex, $G_e$--cocompact set containing $Q(e)$. Thus it contains $C(e)$, which is also $G_e$--cocompact.
\end{proof}

In the next proposition, we deal with the vertex groups.
We replace $Q(v)$ with a larger set $C(v)$.  We first show that $G_v$ acts cocompactly on $C(v)$.  Afterwards we will complete the proof of Theorem~\ref{thm:ConvexVertexGroup} by showing that $C(v)$ is a convex subspace of $X$.

\begin{prop}
For each vertex $v \in T$, let $C(v)$ be the union of $Q(v)$ together with all sets $C(e)$ such that $e$ is adjacent to $v$ in $T$.
Then $G_v$ acts cocompactly on $C(v)$.
\end{prop}

\begin{proof}
Since $G$ is finitely generated, we could have chosen $T$ to be a tree on which $G$ acts with quotient a finite graph (using standard Bass--Serre theory techniques).
Instead, we will use the proof of Lemma~\ref{lem:CocompactEdge} to show that each vertex in the quotient graph $G\backslash T$ has finite valence.
For each vertex $v \in T$, the action of $G_v$ permutes the subspaces $Q(e)$ such that $e$ is adjacent to $v$.  Since these subspaces form a locally finite family, Proposition~\ref{prop:pigeonhole} implies that they lie in finitely many $G_v$--orbits.
It follows that there are finitely many $G_v$--orbits of sets $C(e)$ where $e$ is adjacent to $v$.

Let $C(e_1),\dots,C(e_\ell)$ contain one from each $G_v$--orbit of the sets $C(e)$.
Let $K_i$ be a compact set whose $G_{e_i}$--translates cover $C(e_i)$.  Let $K$ be a compact set whose $G_v$--translates cover $Q(v)$.
Then $C(v)$ is covered by the $G_v$--translates of the compact set $K \cup K_1 \cup \cdots \cup K_\ell$.
\end{proof}

We now complete the proof of Theorem~\ref{thm:ConvexVertexGroup} by showing that the subspace $C(v)$ constructed above is convex, which is the most delicate part of the argument.  

\begin{proof}[Proof of Theorem~\ref{thm:ConvexVertexGroup}]
Fix a vertex $v \in T$.  We will show that $C(v)$ is convex in $X$. Let $e$ be any edge adjacent to $v$.  Recall that $Q(e)$ is defined to be $\pi^{-1}(m_e)$, where $m_e$ is the midpoint of the edge $e$.  Notice that $m_e$ splits the tree $T$ into two halfspaces.  Thus $Q(e)$ can be considered as a ``wall'' in $X$ that separates $X$ into two halfspaces, each a preimage of a halfspace of $T$. The key property of this separation that we need is that any path in $X$ from one halfspace of $Q(e)$ to the other must intersect their separating wall $Q(e)$.   (These halfspaces in $X$ are preimages of connected sets, so they need not be connected.  This will not cause any problems for our proof.)  

Let $H(e)$ be the closed halfspace of $T$ bounded by the midpoint $m_e$ and pointing away from the vertex $v$.
Let $O(e) \subset X$ be the intersection $C(e) \cap \pi^{-1} \bigl( H(e) \bigr)$.

Our strategy is to first prove that $C'(v)$ is convex, where $C'(v)$ is the union of $Q(v)$ with all sets $O(e)$ such that $e$ is adjacent to $v$.
We will complete the proof by showing that $C'(v)$ is actually equal to $C(v)$.

Choose two points $p,q \in C'(v)$.  We claim that the geodesic $c$ in $X$ from $p$ to $q$ lies inside $C'(v)$.  We have several cases depending on which subspaces they are chosen from.

\emph{Case~1}. Suppose $p,q \in Q(v)$.
Any maximal subsegment $c'$ of $c$ outside $Q(v)$ has its endpoints in $Q(e)$ for some $e$ adjacent to $v$.  By convexity, $c'$ must be contained in $C(e)$.  But $c'$ is outside $Q(v)$, so it must lie inside $O(e)$.  Therefore $c$ lies entirely inside $C'(v)$.

\emph{Case~2}. Suppose $p,q \in O(e)$ for some $e$ adjacent to $v$.  Any maximal subsegment $c'$ of $c$ outside $O(e)$ has its endpoints in $Q(e)$.  Therefore $c'$ is a path with both endpoints in $Q(v)$.  By Case~1, the path $c'$ lies in $C'(v)$.  Therefore $c$ lies in $C'(v)$ as well.

\emph{Case~3}. Suppose $p \in O(e)$ and $q \in Q(v)$.  By separation, $c$ contains a point $r \in Q(e)$.  By the first two cases, the subsegments $[p,r]$ and $[r,q]$ each lie in $C'(v)$.  Therefore we are done.

\emph{Case~4}. Suppose $p \in O(e)$ and $q \in O(e')$ for $e\ne e'$ two edges adjacent to $v$.
In this case, $p$ and $q$ are separated by two walls, $Q(e)$ and $Q(e')$.  As in the previous case, we can choose $r$ and $r'$ on $c$ such that $r \in Q(e)$ and $r' \in Q(e')$.
By the first two cases, the subsegments $[p,r]$, $[r,r']$ and $[r',q]$ each lie in $C'(v)$.

We have shown above that $C'(v)$ is convex.
We observe that $C'(v)$ is closed, since it is equal to the union of a locally finite family of closed sets.
Now let us see that $C'(v)$ is equal to our original $C(v)$.
Clearly $C'(v) \subseteq C(v)$.  Recall that $C(e)$ is the closure of the convex hull on $Q(e)$.  By the above argument each $C(e)$ is contained in the closed, convex set $C'(v)$.  Therefore $C(v) \subseteq C'(v)$, as needed.
\end{proof}

\begin{thm}
\label{thm:ComponentCAT0}
Let $G$ be a one-ended group acting geometrically on a $\CAT(0)$ space with isolated flats.
Let $\mathcal{G}$ be any peripheral splitting of $G$ with respect to the natural peripheral structure $\P$.
Then each component vertex group $H$ of the splitting $\mathcal{G}$ acts geometrically on a $\CAT(0)$ space with isolated flats.
Furthermore, the inclusion $H \hookrightarrow G$ induces a topological embedding $\boundary H \hookrightarrow \boundary G$.
\end{thm}

\begin{proof}
Let $\mathbb{Q}$ be the family of infinite subgroups of $H$ of the form $H \cap P$ for $P \in \P$.  By a result of Bowditch \cite{Bowditch01}, we get that $(H,\mathbb{Q})$ is relatively hyperbolic with respect to a family of virtually abelian subgroups.
In any peripheral splitting of $G$, the edge groups are virtually abelian.  By the Flat Torus Theorem, it follows that all edge groups of $\mathcal{G}$ are convex subgroups of $G$.
It follows from Theorem~\ref{thm:ConvexVertexGroup} that the vertex group $H$ acts geometrically on a $\CAT(0)$ space $Y$, which is obtained as a closed convex subspace of $X$.
The space $Y$ has isolated flats by Theorem~\ref{thm:HKConverse}.
Since $Y$ is convex in $X$, the visual boundary $\boundary H = \boundary Y$ embeds in $\boundary X = \boundary G$
\end{proof}

We obtain the following corollary by combining Theorem~\ref{thm:ComponentCAT0} with Theorem~\ref{thm:BowditchComponent}.

\begin{cor}
\label{cor:ComponentCAT0}
Let $G$ be a one-ended group acting geometrically on a $\CAT(0)$ space with isolated flats.
Suppose the maximal peripheral splitting $\mathcal{G}$ is locally finite.
Then each component vertex group $H$ of $\mathcal{G}$ is a one-ended $\CAT(0)$ group with isolated flats that is atomic. \qed
\end{cor}

\section{Local connectivity at rank one points}
\label{sec:RankOneLocalCon}

In this section, we begin the proof of Theorem~\ref{thm:MainThm}.  Our goal is to study the boundary of a one-ended $\CAT(0)$ group with isolated flats.  We need to show that if the maximal peripheral splitting is locally finite, then the boundary of the $\CAT(0)$ space is locally connected.
The first step of the proof, which is the main goal of the current section, 
is to show that the visual boundary $\boundary X$ is locally connected at any point not in the boundary of a flat, i.e., the rank one points.

We use decomposition theory to prove Corollary~\ref{cor:QuotientUSC}, which states that the map $\boundary X \to \boundary (G,\P)$ given by Theorem~\ref{thm:BoundaryQuotient} from the visual boundary to the Bowditch boundary is upper semicontinuous.
Using this structure we are able to pull back local connectivity from the Bowditch boundary to the $\CAT(0)$ boundary, but only at rank one points (see Corollary~\ref{cor:RankOneLC}).

Many of our techniques in the proof of the main theorem depend on elementary results from decomposition theory, which we summarize below.
We refer the reader to \cite{Daverman86} for more details.
Recall that a \emph{decomposition} $\mathcal{D}$ of a topological space $M$ is a partition of $M$.  Decompositions are equivalent to quotient maps in the following sense.
A decomposition $\mathcal{D}$ of a space $M$ has a natural quotient map $\pi\colon M \to M/\mathcal{D}$ obtained by collapsing each equivalence class of $\mathcal{D}$ to a point and endowing the result with the quotient topology.
Conversely, any quotient map $M \to N$ gives rise to an associated decomposition of $M$ consisting of the family of point preimages.

\begin{defn}
Let $\mathcal{D}$ be a decomposition of a space $M$, and let $S$ be a subset of $M$.
The \emph{$\mathcal{D}$--saturation} of $S$ is the union of $S$ together with all sets $D \in \mathcal{D}$ that intersect $S$.
\end{defn}

\begin{defn}
A decomposition $\mathcal{D}$ of a Hausdorff space $M$ is \emph{upper semicontinuous} if each $D \in \mathcal{D}$ is compact, and if, for each $D \in \mathcal{D}$ and each open subset $U$ of $M$ containing $D$, there exists another open subset $V$ of $M$ containing $D$ such that the $\mathcal{D}$--saturation of $V$ is contained in $U$.  In other words, every $D' \in \mathcal{D}$ intersecting $V$ is contained in $U$.
A quotient map is \emph{upper semicontinuous} if its associated decomposition is upper semicontinuous.
\end{defn}

The proof of the following result is left as a routine exercise to the reader (\emph{cf.}\ Proposition~I.1.1 of \cite{Daverman86}).

\begin{prop}
\label{prop:TLSaturationIsClosed}
Let $\mathcal{D}$ be a decomposition of a Hausdorff space $M$. The following are equivalent:
   \begin{enumerate}
   \item $\mathcal{D}$ is upper semicontinuous.
   \item For each open set $U$ of $M$, let $U^*$ equal the union of all members of $\mathcal{D}$ contained in $U$. Then $U^*$ is open.
   \item For each closed set $C$ of $M$, the $\mathcal{D}$--saturation of $C$ is closed. \qed
   \end{enumerate}
\end{prop}

\begin{prop}
\label{prop:Monotone}
Let $\mathcal{D}$ be an upper semicontinuous decomposition of a space $M$. Let $\pi\colon M \to M/\mathcal{D}$ be the natural quotient map.
\begin{enumerate}
  \item \cite[Prop.~I.3.1]{Daverman86} $\pi$ is a proper map; \emph{i.e.}, a set $C \subseteq M/\mathcal{D}$ is compact if and only if $\pi^{-1}(C)$ is compact.
  \item \cite[Prop;~I.4.1]{Daverman86} Suppose each member of $\mathcal{D}$ is connected.
  A set $C \subseteq M/\mathcal{D}$ is connected if and only if $\pi^{-1}(C)$ is connected.
\end{enumerate}
\end{prop}

\begin{prop}
\label{prop:PullBackLC}
Let $\mathcal{D}$ be an upper semicontinuous decomposition of $M$.
Suppose each member of $\mathcal{D}$ is connected.
Let $\{x\}$ be a singleton member of $\mathcal{D}$.
If $M/\mathcal{D}$ is locally connected at the point $\pi(x)$, then $M$ is locally connected at $x$.
\end{prop}

\begin{proof}
Let $U$ be a neighborhood of $x$ in $M$.
By Proposition~\ref{prop:TLSaturationIsClosed} there is a saturated neighborhood $U^*$ of $x$ contained in $U$.
Since $M/\mathcal{D}$ is locally connected at $\pi(x)$, the open set $\pi(U^*)$ contains a connected open neighborhood $\bar{V}$ of $\pi(x)$.
By Proposition~\ref{prop:Monotone}, the preimage $\pi^{-1}(\bar{V})$ is a connected open neighborhood of $x$ contained in $U$.
\end{proof}

\begin{defn}[Subspace decomposition]
Let $\mathcal{D}$ be a decomposition of a Hausdorff space $M$.  Let $W$ be an open $\mathcal{D}$--saturated subspace of $M$. The \emph{induced subspace decomposition} of $W$ is the decomposition consisting of all members of $\mathcal{D}$ that are contained in $W$.
If $\mathcal{D}$ is upper semicontinuous, then the induced subspace decomposition is as well.
\end{defn}

\begin{prop}[\cite{Daverman86}, Prop.~I.2.2]
\label{prop:Metrizable}
If $M$ is compact, metrizable and $\mathcal{D}$ is an upper semicontinuous decomposition of $M$, then $M/\mathcal{D}$ is metrizable.
\end{prop}

\begin{defn}
A collection of subsets $\mathcal{A}$ in a metric space is a \emph{null family} if, for each $\epsilon >0$, only finitely many of the sets $A \in \mathcal{A}$ have diameter greater than $\epsilon$.
\end{defn}

Note that if $M$ is compact and metrizable, then being a null family does not depend on the choice of metric on $M$.

The following property of null families is related to the definition of upper semicontinuous, but here we do not require that the members of $\mathcal{A}$ be disjoint.

\begin{prop}
\label{prop:NullAlmostUSC}
Let $\mathcal{A}$ be a null family of compact sets in a metric space $M$.
Suppose $q \in M$ is not contained in any member of the family $\mathcal{A}$.
Then each neighborhood $U$ of $q$ contains a smaller neighborhood $V$ of $q$ such that each $A \in \mathcal{A}$ intersecting $V$ is contained in $U$.
\end{prop}

\begin{proof}
Let $U$ be a neighborhood of $q$, and suppose $B(q,\epsilon) \subseteq U$.
Choose $\delta$ such that $0<\delta<\epsilon/2$ and such that $d(q,A) > \delta$ for each of the finitely many $A \in \mathcal{A}$ with diameter greater than $\epsilon/2$.
The result follows if we set $V = B(q,\delta)$.
\end{proof}

\begin{rem}
\label{rem:NullConeTopology}
The notion of a null family can be formulated in terms of the cone topology on $\boundary X$ as follows. Fix a basepoint $x_0 \in X$.  A collection $\mathcal{A}$ of subspaces of $\boundary X$ is a null family provided that there exists $D > 0$ such that for each $r<\infty$ only finitely many members of the collection $\mathcal{A}$ are not contained in any set of the form $U(\cdot,r,D)$.  It follows from Definition~\ref{defn:BoundaryMetric} that only finitely many members of $\mathcal{A}$ have diameter at least $1/r$ with respect to the metric $d_D$ on $\boundary_{x_0}X$.

A similar condition can be used to characterize null families in the cone topology on $\bar{X} = X \cup \boundary X$.  We leave the proof as an exercise for the reader.
\end{rem}

When the collection $\mathcal{A}$ of subsets of $M$ is disjoint, there is an associated decomposition of $M$ consisting of the sets in $\mathcal{A}$ together with all singletons $\{x\}$ such that $x \in M - \bigcup \mathcal{A}$.  By a slight abuse of notation, we let $M/\mathcal{A}$ denote the corresponding quotient in which each member of $\mathcal{A}$ is collapsed to a point.

Decompositions arising from null families play a central role in the proof of the main theorem.  The following result is stated as an exercise in \cite{Daverman86}.  The proof is nearly identical to the proof of Proposition~\ref{prop:NullAlmostUSC}.

\begin{prop}[\cite{Daverman86}, Prop.~I.2.3]
\label{prop:NullUSC}
Let $\mathcal{A}$ be a null family of disjoint compact subsets in a metric space $M$.  Then the associated decomposition of $M$ is upper semicontinuous.
\end{prop}

\begin{prop}
Let $X$ be a $\CAT(0)$ space that has isolated flats with respect to the family of flats $\mathcal{F}$.
Let $\mathcal{A}$ be the family of spheres $\set{\boundary F}{F \in \mathcal{F}}$.
Then $\mathcal{A}$ is a null family of disjoint compact subsets in $\boundary X$.
\end{prop}

\begin{proof}
Each $\boundary F \in \mathcal{A}$ is a sphere, which is compact.  The definition of isolated flats immediately implies that the members of $\mathcal{A}$ are pairwise disjoint.  Thus we only need to show they are a null family.

Choose a basepoint $x_0 \in X$.  Let $\kappa$ be the constant from Theorem~\ref{thm:facts}.  We will prove the following claim below: for any flat $F \in \mathcal{F}$ satisfying $d(x_0,F) \ge r + 3\kappa$ for some constant $r$, there exists a geodesic ray $c$ based at $x_0$ such that $\boundary F \subseteq U(c,r,7\kappa)$.

Since the collection of flats $\mathcal{F}$ is locally finite, there are only finitely many within a distance $r+3\kappa$ of $x_0$ for each $r<\infty$. Thus it will follow from Remark~\ref{rem:NullConeTopology} that $\mathcal{A}$ is a null family.

In order to prove the claim, let $q$ be the nearest point in $F$ to $x_0$.  Then $d(x_0,q) \ge r+3\kappa$.  Let $c,c'$ be geodesic segments from $x_0$ to $F$. By Theorem~\ref{thm:facts} the set $[x_0,q] \cup F$ is $\kappa$--quasiconvex.  Thus there exist $s,s'$ with $c(s),c'(s')$ both contained in $\bignbd{[x_0,q]}{\kappa}\cap \nbd{F}{\kappa}$.  It follows that $d\bigl(c(s),q\bigr)$ and $d\bigl(c'(s'),q\bigr)$ are each less than $3\kappa$.  Thus $d\bigl(c(s),c'(s')\bigr) < 6\kappa$.  By the Law of Cosines $d\bigl(c(r),c'(r)\bigr)<6\kappa$.  In particular $c' \in U(c,r,6\kappa)$.

Now suppose $c$ and $c'$ are geodesic rays asymptotic to $F$ (possibly not intersecting $F$).  Then each is a limit of geodesic segments that intersect $F$.  In this case, we conclude that $c' \in U(c,r,7\kappa)$.  Therefore $\boundary F \subseteq U(c,r,7\kappa)$ for any $c$ with $c(\infty) \in \boundary F$, establishing the claim.
\end{proof}

\begin{cor}
\label{cor:QuotientUSC}
Let $G$ act geometrically on a $\CAT(0)$ space $X$ with isolated flats.
Let $\P$ be the standard relatively hyperbolic structure on $G$.
Then the quotient map $\boundary X \to \boundary X / \mathcal{A} \to \boundary (G,\P)$ given by Theorem~\ref{thm:BoundaryQuotient} is upper semicontinuous. \qed
\end{cor}

\begin{cor}
\label{cor:RankOneLC}
Let $G$ be a one-ended group acting geometrically on a $\CAT(0)$ space $X$ with isolated flats.
Then $\boundary X$ is locally connected at any point $\xi$ not in the boundary of any flat.
\end{cor}

\begin{proof}
By Theorem~\ref{thm:Bowditch}, the Bowditch boundary $\boundary (G,\P)$  is locally connected at every point. Each member of the decomposition $\mathcal{A}$ is either a point or a sphere $S^k$ with $k>0$.  Thus all members of $\mathcal{A}$ are connected.
The result follows immediately from Proposition~\ref{prop:PullBackLC}.
\end{proof}

\section{Local connectivity on the boundary of a flat}
\label{sec:FlatLocalCon}

Recall our main goal is to establish local connectivity of the visual boundary of a  $\CAT(0)$ group with isolated flats in the setting where the maximal peripheral splitting is locally finite.

In this section we focus on the special case where the maximal peripheral splitting is trivial, in other words, we study atomic groups. 
The main goal of this section is to prove Theorem~\ref{thm:UnsplittableLC}, which states that the boundary of an atomic group is always locally connected.

In the previous section, we showed that the boundary of such a group is locally connected at any point not in the boundary of a flat subspace.  To reach our goal, it suffices to show that the boundary $\boundary X$ is weakly locally connected at points of $\boundary F$, where $F$ is a flat subspace.
Recall that a space is \emph{weakly locally connected} at a point $\xi$ if $\xi$ has a local base of (not necessarily open) connected neighborhoods.
A space is locally connected if it is weakly locally connected at every point.

In order to understand the topology of $\boundary X$ near a point of $\boundary F$, we partition $\boundary X$ into $\boundary F$ and its complement $\Upsilon = \boundary X - \boundary F$.
If $P$ is the stabilizer of $F$, then we will see that $P$ acts properly and cocompactly on $F$, on $\Upsilon$, and also on $\boundary(G,\P) - \{\rho\}$, where $\rho$ is the parabolic point corresponding to $\boundary F$.  Our main strategy is to exploit similarities between these three spaces.
Many of these similarities do not not depend on the extra hypothesis that $G$ does not peripherally split.
After developing some features of this similarity, we will add the extra ``atomic'' hypothesis, which implies that both $F$ and $\Upsilon$ are $0$--connected spaces.
Since they share proper and cocompact group actions by the same group, we are able to transfer $0$--connectedness properties between them.
In particular, the local connectedness of $\boundary X = \Upsilon \cup \boundary F$ at a point $\xi \in \boundary F$ will follow from the local connectedness of $\bar{F} = F \cup \boundary F$ at $\xi$.

The similarity between $F$ and $\Upsilon$ was first introduced and extensively studied by Haulmark \cite{HaulmarkCAT0} in the general situation of groups acting on $\CAT(0)$ spaces with isolated flats (without the ``atomic'' hypothesis).
This similarity depends heavily on the following two lemmas that play a significant role in Haulmark's work.

\begin{lem}[\cite{HaulmarkCAT0}]
\label{lem:CompactTransfer}
Let $G$ be a group acting geometrically on a $\CAT(0)$ space $X$ with isolated flats with respect to the family $\mathcal{F}$.  Let $F \in \mathcal{F}$ be a flat with stabilizer $P$, and let $\Upsilon = \boundary X - \boundary F$.
For each compact set $K \subset \Upsilon$ there exists a compact set $C \subset F$ such that for each $\eta \in K$ there is a geodesic ray $c'$ with $c'(0) \in C$ and $c'(\infty) =\eta$ such that $c'$ meets $F$ orthogonally.
Furthermore the compact set $C$ can be chosen $P$--equivariantly in the sense that if $p \in P$ then $pC$ is the compact set of $F$ corresponding to $pK$.
\end{lem}

We occasionally apply the previous lemma in the following special case: each point of $\Upsilon$ is the endpoint of a geodesic ray meeting $F$ orthogonally.

The following corollary of Theorem~\ref{thm:facts} was first observed by Haulmark.

\begin{lem}[\cite{HaulmarkCAT0}]
\label{lem:SmallTransfer}
Let $\kappa$ be the constant given by Theorem~\ref{thm:facts}.
Let $\Upsilon = \boundary X - \boundary F$, and suppose $c'$ is a geodesic ray meeting $F$ orthogonally.
Suppose $c$ is a geodesic ray contained in $F$.
If $c'(0) \in U(c,r,D)$ for some constants $r$ and $D$ then $c'(\infty) \in U(c,r,D + \kappa)$.
Conversely if $c'(\infty) \in U(c,r,D)$ then $c'(0) \in U(c,r,D + \kappa)$.
\end{lem}

In any $\CAT(0)$ space $X$ with a geometric group action, the family of translates of a compact fundamental domain is a null family in the following sense.

\begin{prop}[\cite{Bestvina96}]
\label{prop:YNullTranslates}
Let $H$ be any group acting geometrically on a $\CAT(0)$ space $Y$.
Let $C \subset Y$ be any compact set.
Then the collection of $H$--translates of $C$ is a null family in the compact space $\bar{Y}$.
\end{prop}

We will use the previous proposition in the case when $Y$ is a flat subspace $F$ of a $\CAT(0)$ space with isolated flats, and $H$ is its stabilizer $P$.

The similarity between $F$ and $\Upsilon$ allows us to transfer this version of the null condition to the action of $P$ on $\Upsilon$ as follows:

\begin{prop}
\label{prop:UpsilonNull}
Let $G$ act geometrically on a space $X$ with isolated flats. Choose $F \in \mathcal{F}$ with stabilizer $P$, and let $\Upsilon = \boundary X - \boundary F$.
If $K \subset \Upsilon$ is compact, then the collection of $P$--translates of $K$ is a null family in $\bar{\Upsilon} = \boundary X$.
\end{prop}

\begin{proof}
Choose a compact set $K \subset \Upsilon$.
Our strategy is to exploit the similarity between $\Upsilon$ and $F$ as follows:
use Lemma~\ref{lem:CompactTransfer} to pull $K$ back to a compact set $C$ in $F$, use Proposition~\ref{prop:YNullTranslates} to see that almost all translates of $C$ are ``small'' in $\bar{F}$, and then apply Lemma~\ref{lem:SmallTransfer} to conclude that the corresponding translates of $K$ are similarly small in $\bar{\Upsilon} = \boundary X$.

For our given compact set $K \subset \Upsilon$, let $C$ be the corresponding compact set of $F$ given by Lemma~\ref{lem:CompactTransfer}.  
Fix a positive constant $D$, and let $\kappa$ be the constant from Theorem~\ref{thm:facts}.
By Remark~\ref{rem:NullConeTopology}, it suffices to show that for each $r < \infty$ only finitely many $P$--translates of $K$ are not contained in any set of the form $U(\cdot,r,D+\kappa)$.
According to Proposition~\ref{prop:YNullTranslates}, the $P$--translates of $C$ are a null family. Thus only finitely many $P$--translates of $C$ are not contained in a set of the form $U(\cdot,r,D)$.
For any $p \in P$, if $pC$ lies in a set of the form $U(\cdot,r,D)$ then the corresponding set $pK$ of $\Upsilon$ lies in a set of the form $U(\cdot,r,D+\kappa)$ by Lemma~\ref{lem:SmallTransfer}. So $\{pK\}$ is a null family in $\bar{\Upsilon}$.
\end{proof}

Since $P$ acts cocompactly on $F$, Haulmark exploited the similarity between $F$ and $\Upsilon$ to show that $P$ also acts cocompactly on $\Upsilon$ \cite{HaulmarkCAT0}.

For the rest of this section, we focus on the special setting where $G$ is atomic.
In this special situation, we can improve the conclusion of Haulmark's cocompactness theorem to get a compact, connected fundamental domain for the action of $P$ on $\Upsilon$.

\begin{prop}
\label{prop:ConnectedFD}
Let $G$ be a one-ended group acting geometrically on a $\CAT(0)$ space $X$ that has isolated flats with respect to the family of flats $\mathcal{F}$.
Suppose $G$ is atomic.
For each flat $F \in \mathcal{F}$ with stabilizer $P$, there is a compact connected set $K$ in $\Upsilon = \boundary X - \boundary F$ whose $P$--translates cover $\Upsilon$.

Furthermore if $\mathcal{T}$ is any finite generating set for the group $P$, we can choose $K$ large enough that $K$ intersects $tK$ for all $t \in \mathcal{T}$.
\end{prop}

The previous results of this section are proved using the similarity between $\Upsilon$ and $F$.
However for the proof of Proposition~\ref{prop:ConnectedFD}, we  use the other similarity between $\Upsilon$ and $\boundary(G,\P) - \{\rho\}$, where $\rho$ is the parabolic point corresponding to $\boundary F$.
The proof of the proposition relies on the fact that each parabolic point $\rho$ of the Bowditch boundary is bounded parabolic; i.e., there is a compact fundamental domain for the action of its stabilizer on $\boundary(G,\P) - \{\rho\}$. (See Definition~\ref{def:RelHyp}.)

We also need the following lemma, which allows us to increase any such compact fundamental domain to a connected one, provided that $G$ is atomic.

\begin{lem}
\label{lem:CompactConnectedRelative}
Let $(G,\P)$ be relatively hyperbolic.  Suppose $G$ is atomic and each $P \in \P$ is finitely presented and does not contain an infinite torsion subgroup.  Let $\rho \in \boundary(G,\P)$ be a parabolic point with stabilizer $P$.  Let $C_0$ be any compact fundamental domain for the action of $P$ on $\boundary(G,\P) - \{\rho\}$.
Then $C_0$ is contained in a compact connected fundamental domain $C$.
\end{lem}

\begin{proof}
By Theorem~\ref{thm:Bowditch} the Bowditch boundary $\boundary(G,\P)$ is connected and locally connected, and the parabolic point $\rho$ is not a global cut point.  Thus $\boundary(G,\P) - \{\rho\}$ is an open, connected subset of the compact, locally connected, metrizable space $\boundary(G,\P)$.
It follows that $\boundary(G,\P) - \{\rho\}$ is path connected by \cite[31C.1]{Willard70}.

Let $d$ be a metric on $\boundary(G,\P)$, and let $\epsilon = d(\rho,C_0)$.  We can cover the compact set $C_0$ by finitely many open connected sets with diameter less then $\epsilon/2$.  In $\boundary(G,\P)$ the union of the closures of these sets is a compact set $C_1$ containing $C_0$ and having only finitely many components.  By our choice of $\epsilon$, the compact set $C_1$ is contained in $\boundary(G,\P) - \{\rho\}$.

Finally we form $C$ from $C_1$ by attaching finitely many compact paths in $\boundary(G,\P) - \{\rho\}$ that connect the finitely many components of $C_1$.
\end{proof}

\begin{proof}[Proof of Proposition~\ref{prop:ConnectedFD}]
Our strategy is to use the quotient map $\pi \colon \boundary X \to \boundary(G,\P)$ given by Theorem~\ref{thm:BoundaryQuotient}.  We will find an appropriate fundamental domain in the Bowditch boundary and then pull it back via $\pi$ to get a compact connected fundamental domain in $\boundary X$.  

Let $\rho$ be the parabolic point of $\boundary(G,\P)$ stabilized by $P$; i.e., $\{\rho\}$ is the image of $\boundary F$ in the Bowditch boundary.  By the definition of relative hyperbolicity, the action of $P$ on $\boundary(G,\P) - \{\rho\}$ has a compact fundamental domain  $C_0$.  Increasing the size of $C_0$, we may assume without loss of generality that $C_0$ intersects the finitely many translates $tC_0$ for all $t \in \mathcal{T}$.  By Lemma~\ref{lem:CompactConnectedRelative}, we can increase $C_0$ to a compact, connected fundamental domain $C$ intersecting its translates $tC$ for all $t \in \mathcal{T}$.

Recall that the quotient $\pi \colon \boundary X \to \boundary(G,\P)$ collapses connected sets to points; i.e. each member of the associated decomposition of $\boundary X$ is either a point or the boundary of a one-ended peripheral subgroup (in our case this boundary is a sphere).
By Corollary~\ref{cor:QuotientUSC}, the quotient $\pi$ is upper semicontinuous.
It follows from Proposition~\ref{prop:Monotone} that the preimage $K = \pi^{-1}(C)$ is compact and connected.
Theorem~\ref{thm:BoundaryQuotient} implies that $\pi$ is $G$--equivariant.  Thus $K$ is a fundamental domain for the action of $P$ on $\Upsilon$, and $K$ intersects its translates $tK$ for each $t \in \mathcal{T}$.
\end{proof}

Our goal for the rest of this section is to prove the following proposition. 

\begin{prop}
\label{prop:FlatLC}
Let $G$ be a one-ended group acting geometrically on a $\CAT(0)$ space with isolated flats.
Assume $G$ is atomic.
Then $\boundary X$ is weakly locally connected at any point in the boundary of any flat.
\end{prop}

The proof of Proposition~\ref{prop:FlatLC} depends on three lemmas.
Before discussing the lemmas, we outline the broad strategy that leads to the proof.
Recall that $\boundary X = \bar\Upsilon$ is similar in many ways to $\bar{F}$.
We know that $\bar{F}$ is locally connected at each point $\xi \in \boundary F$ by Proposition~\ref{prop:BarXLC}.
In order to prove that $\bar\Upsilon$ is also locally connected at $\xi$, we will describe a procedure for transferring small connected neighborhoods of $\xi$ from $\bar{F}$ to $\bar\Upsilon$, which is valid when $G$ does not peripherally split.

The atomic hypothesis implies that $\Upsilon$ is connected.
Obviously $F$ is also connected.
The foundation of our strategy is the following lemma, which allows us to transfer $0$--connectedness from $F$ to $\Upsilon$ using the fact that both $F$ and $\Upsilon$ are $0$--connected spaces on which the same group $P$ acts properly and cocompactly.

\begin{lem}
\label{lem:ConnectedTransfer}
Assume $G$ is atomic.
There exist compact, connected fundamental domains $C \subset F$ and $K \subset \Upsilon$ for the actions of $P$ on each, such that the following holds.
Let $\mathcal{P}$ be any subset of the group $P$.
\[
   \text{If $\bigcup_{p \in \mathcal{P}} pC$ is connected in $F$, then $\bigcup_{p \in \mathcal{P}} pK$ is connected in $\Upsilon$.}
\]
\end{lem}

\begin{proof}
Choose a compact connected fundamental domain $C$ for the action of $P$ on $F$.  Let $\mathcal{T}$ be the set of elements $t \in P$ such that $C$ intersects $tC$. Then $\mathcal{T}$ is a finite generating set for $P$.
Choose a compact connected fundamental domain $K$ for the action of $P$ on $\Upsilon$ as given by Proposition~\ref{prop:ConnectedFD} such that $K$ intersects $tK$ for all $t \in \mathcal{T}$.
This intersection property immediately implies the following condition that holds for all $p \in P$:
\[
   \text{If $C \cap pC$ is nonempty, then $K \cap pK$ is nonempty.}
\]
This condition easily implies our conclusion.
\end{proof}

The following terminology and notation will be used throughout the rest of this section and the eventual proof of Proposition~\ref{prop:FlatLC}.
Let $C$ and $K$ be the compact, connected fundamental domains given by the previous lemma.
By Lemma~\ref{lem:CompactTransfer} there exists a geodesic ray $c'$ in $X$ meeting $F$ orthogonally. We will treat the points $q_0 = c'(0) \in F$ and $q_\infty = c'(\infty) \in \Upsilon$ as basepoints in $F$ and $\Upsilon$ respectively.
Translating $C$ and $K$ by the cocompact group actions, we can also assume that $q_0 \in C$ and $q_\infty \in K$.

Suppose $\xi \in \boundary F$.  As mentioned above, our strategy for proving Proposition~\ref{prop:FlatLC} is to transfer small connected neighborhoods of $\xi$ in $\bar{F}$ to small connected neighborhoods of $\xi$ in $\bar\Upsilon$.
To facilitate this transfer, we assume that the given neighborhood $\bar{N}$ of $\xi$ in $\bar{F}$ is \emph{clean} in the sense that $N = \bar{N} \cap F$ is connected, and each point of $\bar{N}$ is a limit point of $N$.
Recall that $\xi$ has a local base of clean connected neighborhoods by Proposition~\ref{prop:BarXLC}.

For each clean connected neighborhood $\bar{N}$, we will define a corresponding set $\bar{Z}$ in $\bar\Upsilon = \boundary X$.  In the two subsequent lemmas, we will show that $\bar{Z}$ is a connected neighborhood of $\xi$ in $\Upsilon$, and that $\bar{Z}$ can be chosen arbitrarily small.  We begin with the construction of $\bar{Z}$.

\begin{defn}[Associated neighborhoods]
Suppose $\xi \in \boundary F$ and $\bar{N}$ is a clean connected neighborhood of $\xi$ in $\bar{F}$.
Let $N = \bar{N} \cap F$, and let $\Lambda = \bar{N} \cap \boundary F$.
Let $\mathcal{P}$ be the set of all $p \in P$ such that $pC$ intersects $N$.
The corresponding set $Z \subset \Upsilon$ is the union $\bigcup_{p \in \mathcal{P}} pK$.
Finally the \emph{$\bar\Upsilon$--neighborhood  associated to $\bar{N}$} is the set $\bar{Z} = Z \cup \Lambda$.
\end{defn}

For this definition to make sense, we must verify that $\bar{Z}$ is actually a neighborhood of $\xi$ in $\bar\Upsilon$.  The next lemma establishes that $\bar{Z}$ is, in fact, a (clean) connected neighborhood.

\begin{lem}
\label{lem:ConnectedNeighborhood}
For any clean, connected neighborhood $\bar{N}$ of $\xi$ in $\bar F$, the $\bar\Upsilon$--neighborhood $\bar{Z}$ associated to $\bar{N}$, defined above, is a connected neighborhood of $\xi$ in $\bar\Upsilon$.
\end{lem}

\begin{proof}
We first verify that $\bar{Z}$ is a neighborhood of $\xi$ in $\bar\Upsilon$.
Let $\kappa$ be the constant given by Theorem~\ref{thm:facts}, and let $D>0$ be an arbitrary constant.
Choose a geodesic ray $c$ in $F$ with $c(\infty) = \xi$.
Since $\bar N$ is a neighborhood of $\xi$ in $\bar{F}$, we can choose $R$ large enough so that $U(c,R,D) \cap \bar{F}$ lies inside $\bar N$.
By Proposition~\ref{prop:UpsilonNull}, the collection of $P$--translates of our compact fundamental domain $K$ is a null family in $\bar\Upsilon$.
Therefore there exists a neighborhood $V$ of $\xi$ in $\bar\Upsilon$ such that $V \subseteq U(c,R,D+\kappa)$, and such that every $pK$ intersecting $V$ is contained in $U(c,R,D+\kappa)$ by Proposition~\ref{prop:NullAlmostUSC}.

It follows that $V \subseteq \bar{Z}$. 
Indeed, each element $\eta \in V$ either lies in $\Upsilon$ or in $\boundary F$.
In the first case, by our choice of $K$, the point $\eta$ lies in $pK$ for some $p \in P$.
Each such $pK$ is contained in $U(c,R,D+\kappa)$.
In particular,
$p ( q_\infty)$ lies in $U(c,R,D+\kappa)\subset\bar\Upsilon$.
Thus $p(q_0) \in U(c,R,D) \subset \bar{F}$ by Lemma~\ref{lem:SmallTransfer}.
Since $p(q_0) \in pC$,
our choice of $R$ implies that the $P$--translate $pC$ intersects $N$.
By the definition of $Z$, we have $\eta\in pK \subset Z \subset \bar{Z}$.

In the second case, we have $\eta \in V \cap \boundary F$, so
\[
   \eta \ \in \ U(c,R,D+\kappa) \cap \boundary F  \ \subseteq\   U(c,R,D) \cap \boundary F \ \subseteq\ \bar{N} \cap \boundary F \ = \ \Lambda \ \subseteq \ \bar{Z}.
\]
Combining the two cases, we see that $V \subseteq \bar{Z}$, as desired.
It follows that $\bar{Z}$ is a neighborhood of $\xi$ in $\bar\Upsilon$.

Next we will see that $\bar{Z}$ is connected.  The cleanliness of $\bar{N}$ means that $N = \bar{N} \cap F$ is connected.  Recall that $\mathcal{P}$ is the set of all $p$ such that $pC$ intersects $N$.
Since $C$ is connected, the union $\hat{N} = \bigcup_{p \in \mathcal{P}} pC$ is also connected.  By Lemma~\ref{lem:ConnectedTransfer}, the union $Z = \bigcup_{p \in \mathcal{P}} pK$ is connected as well.

In order to show that $\bar{Z} = Z \cup \Lambda$ is connected, it suffices to show that every point of $\Lambda$ is a limit point of $Z$.
Since $\bar{N}$ is clean, each point $\zeta \in \Lambda$ is a limit of a sequence $\{x_i\}$ in $N$.  Each $x_i \in p_iC$ for some $p_i \in \mathcal{P}$, and by Proposition~\ref{prop:YNullTranslates} the sequence $\bigl\{ p_i(q_0) \bigr\}$ also converges to $\zeta$.
By Lemma~\ref{lem:SmallTransfer}, the sequence $\bigl\{ p_i(q_\infty) \bigr\}$ converges to $\zeta$ as well.  Since $q_\infty \in K$, we have $p_i (q_\infty) \in p_i K \subseteq Z$.  Thus $\zeta$ is a limit point of $Z$.
\end{proof}

\begin{lem}
\label{lem:ArbitrarilySmall}
The $\bar{\Upsilon}$--neighborhood $\bar{Z}$ associated to $\bar{N}$ can be made arbitrarily small by choosing $\bar{N}$ to be a sufficiently small neighborhood of $\xi$.
\end{lem}

\begin{proof}
Let $U$ be a neighborhood of $\xi$ in $\bar\Upsilon = \boundary X$.  Our goal is to show that if $\bar{N}$ is chosen appropriately, its associated neighborhood $\bar{Z}$ will be contained in $U$.
By Proposition~\ref{prop:UpsilonNull}, there is a neighborhood $V$ of $\xi$ in $\bar\Upsilon$ such that $V \subseteq U$ and every $pK$ intersecting $V$ is contained in $U$.
By Lemma~\ref{lem:SmallTransfer}, there is a neighborhood $W$ of $\xi$ in $\bar{F}$ such that 
$W \cap \boundary F \subset V$ and such that
if $p(q_0) \in W$ then $p(q_\infty) \in V$ and hence $pK \subset U$.
Proposition~\ref{prop:YNullTranslates} gives a neighborhood $W'$ of $\xi$ in $\bar{F}$ such that 
$W' \subseteq W$ and every $pC$ intersecting $W'$ is contained in $W$.
Due to Proposition~\ref{prop:BarXLC}, there is a clean connected neighborhood $\bar{N}$ of $\xi$ inside $W'$.

It follows that the associated $\bar\Upsilon$--neighborhood $\bar{Z}$ is contained in $U$.  Indeed, each element $\eta \in \bar{Z}$ either lies in $Z$ or in $\Lambda=\bar{N} \cap \boundary F$.
Suppose first that $\eta \in Z$.  Then $\eta \in pK$ for some $p$ such that $pC$ intersects $N$, and $N \subseteq W'$.  In this case, it is clear from our choices above that $\eta \in U$.
On the other hand, suppose $\eta\in\Lambda$.  Since $\Lambda \subseteq W \cap \boundary F$, it is contained in $V$, which is contained in $U$.  We have shown that $\bar{Z} \subseteq U$, as needed.
\end{proof}

At this point the proof of Proposition~\ref{prop:FlatLC} is nearly complete.

\begin{proof}[Proof of Proposition~\ref{prop:FlatLC}]
Suppose $\xi \in \boundary F$.  We must show that $\boundary X = \bar\Upsilon$ is weakly locally connected at $\xi$.
The combination of Lemmas 
~\ref{lem:ConnectedNeighborhood} and~\ref{lem:ArbitrarilySmall} implies that $\xi$ has arbitrarily small connected neighborhoods (that are not necessarily open).
\end{proof}

Combining Corollary~\ref{cor:RankOneLC} and Proposition~\ref{prop:FlatLC} completes the proof of Theorem~\ref{thm:UnsplittableLC}.

\section{The limit of a tree system of spaces}
\label{sec:TreeLimits}

In order to complete the proof of Theorem~\ref{thm:MainThm}, we need to examine the boundary of a one-ended $\CAT(0)$ group $G$ with isolated flats whose maximal peripheral splitting is nontrivial and locally finite.

By Corollary~\ref{cor:ComponentCAT0}, the component vertex groups of this splitting are atomic $\CAT(0)$ groups with isolated flats.  Therefore by Theorem~\ref{thm:UnsplittableLC} the boundary of each component vertex group is locally connected.
We will see in Section~\ref{sec:PuttingTogether} that the visual boundary of $G$ is obtained by gluing copies of the component group boundaries along spheres in the pattern of the Bass-Serre tree.  
The tool necessary to make this precise is 
the notion of a tree system of spaces, introduced by \Swiatkowski\ in \cite{SwiatkowskiTreesOfCompacta}.

In this section we present the definition of the limit of a tree system of metric compacta.  We also prove Theorem~\ref{thm:TreeLimLC}, which roughly states that a tree system of locally connected spaces has a limit that is also locally connected.

\begin{defn}
A \emph{tree} is a connected nonempty graph without circuits.  We use Serre's notation for graphs with oriented edges.  A graph has a vertex set $\mathcal{V}$, a set of oriented edges $\mathcal{E}$, a map $\mathcal{E} \to \mathcal{V} \times \mathcal{V}$ denoted $e \mapsto \bigl(o(e), t(e)\bigr)$, and an involution $\mathcal{E} \to \mathcal{E}$ denoted $e \mapsto \bar{e}$.  We require that $\bar{e}\ne e$ and that $o(e)=t(\bar{e})$.  The vertices $o(e)$ and $t(e)$ are the \emph{origin} and \emph{terminus} of the edge $e$.  We refer the reader to \cite{Serre77} for more details.
\end{defn}

\begin{defn}\label{defn:TreeSystem}
A \emph{tree system $\Theta$ of metric compacta} consists of the following data:
  \begin{enumerate}
  \item $T$ is a bipartite tree with a countable vertex set $\mathcal{V} = \mathcal{C} \coprod \mathcal{P}$ such that each vertex $w \in \mathcal{P}$ has finite valence.
  \item To each vertex $v \in \mathcal{V}$ there is associated a compact metric space $K_v$.
  \item To each edge $e \in \mathcal{E}$ there is associated a compact metric space $\Sigma_e$, a homeomorphism $\phi_e \colon \Sigma_e \to \Sigma_{\bar{e}}$ such that $\phi_{\bar{e}}=\phi_e^{-1}$,
and a topological embedding $i_e \colon \Sigma_e \to K_{t(e)}$.
  \item For each $v \in \mathcal{C}$ the family of subspaces $\set{i_e(\Sigma_e)}{t(e)=v}$ is null and consists of pairwise disjoint sets. We will refer to the spaces $K_v$ for $v\in\mathcal{V}$ as \emph{component spaces}. 
  \item For each $w \in \mathcal{P}$ and each $e$ with $t(e)=w$ the map $i_e \colon \Sigma_e \to K_w$ is a homeomorphism. We will refer to the spaces $K_w$ for $w\in\mathcal{P}$ as \emph{peripheral spaces}.
  \end{enumerate}
The tree system $\Theta$ is \emph{degenerate} if the tree $T$ contains only one vertex, and that vertex is peripheral.
\end{defn}

\begin{rem}
The definition above of tree system is slightly more general than the one used by \Swiatkowski\ in \cite{SwiatkowskiTreesOfCompacta}. In the special case that each vertex $w \in \mathcal{P}$ has valence two, our definition is equivalent to \Swiatkowski's using a barycentric subdivision of his tree.  The proofs in \cite{SwiatkowskiTreesOfCompacta} generalize to tree systems in the sense of Definition~\ref{defn:TreeSystem} with only minor modifications.
\end{rem}

Let $\# \Theta$ denote the quotient $\bigl(\coprod_{v \in \mathcal{C} \cup \mathcal{P}} K_v\bigr)/\!\sim$ by the equivalence relation generated by $i_e(x) \sim i_{\bar{e}} \phi_e (x)$ for all edges $e \in \mathcal{E}$ and all $x \in \Sigma_e$
endowed with the quotient topology.

For each subtree $S$ of $T$
let $\mathcal{V}_S = \mathcal{C}_S \coprod \mathcal{P}_S$ and $\mathcal{E}_S$ denote the set of vertices and the set of edges of $S$. Let $\Theta_S$ denote the restriction of the tree system $\Theta$ to the subtree $S$.
Let $N_S=\set{e \in \mathcal{E}}{o(e) \notin \mathcal{V}_S \text{ and } t(e) \in \mathcal{V}_S}$ be the set of oriented edges adjacent to $S$ but not contained in $S$, oriented towards $S$.

\begin{defn}[Limit of a tree system]
\label{defn:TreeLimit}
For each finite subtree $F$ of $T$, the \emph{partial union} $K_F$ is defined to be $\# \Theta_F$.  Since $F$ is finite it follows that $K_F$ is compact and metrizable.
Let $\mathcal{A}_F = \set{i_e(\Sigma_e)}{e \in N_F}$.
We consider $\mathcal{A}_F$ to be a family of subsets of $K_F$, and note that this is a null family that consists of pairwise disjoint compact sets.  Let $K_F^* = K_F / \mathcal{A}_F$, the quotient formed by collapsing each set in $\mathcal{A}_F$ to a point.
By Propositions \ref{prop:NullUSC} and \ref{prop:Metrizable},
the quotient $K_F^*$ is metrizable.

For each pair of finite subtrees $F_1 \subseteq F_2$, let $f_{F_1F_2} \colon K^*_{F_2} \to K^*_{F_1}$ be the quotient map obtained by collapsing $K_s$ to a point for each $s \in \mathcal{V}_{F_2} - \mathcal{V}_{F_1}$ and identifying the resulting quotient space with $K_{F_1}^*$.  Since $f_{F_1 F_2} \of f_{F_2 F_3} = f_{F_1 F_3}$ whenever $F_1 \subseteq F_2 \subseteq F_3$, the system of spaces $K_F^*$ and maps $f_{F F'}$ where $F\subseteq F'$ is an inverse system of metric compacta indexed by the poset of all finite subtrees $F$ of $T$.

The \emph{limit $\lim \Theta$ of the tree system} $\Theta$ is the inverse limit of the above inverse system.  Observe that $\lim\Theta$ is compact and metrizable, since it is an inverse limit of a countable system of compact metrizable spaces.
\end{defn}

A function $f \colon Y \to Z$ is \emph{monotone} if $f$ is surjective and for each $z \in Z$ the preimage $f^{-1}(z)$ is compact and connected.
The following theorem due to Capel is used in the proof of Theorem~\ref{thm:TreeLimLC}.

\begin{thm}[\cite{Capel54}]
\label{thm:Capel}
Let $\{X_\alpha\}$ be an inverse system such that each bonding map $X_\alpha \to X_{\beta}$ is monotone.
If each factor space $X_\alpha$ is compact and locally connected then the inverse limit $\varprojlim X_\alpha$ is locally connected.
\end{thm}

\begin{thm}
\label{thm:TreeLimLC}
The limit $\lim \Theta$ of a nondegenerate tree system $\Theta$ is locally connected, provided that each component vertex space $K_v$ with $v \in \mathcal{C}$ is connected and locally connected and that each peripheral vertex space $K_w$ with $w \in \mathcal{P}$ is nonempty.
\end{thm}

\begin{proof}
Since the component vertex spaces are connected and the peripheral vertex spaces are nonempty, $K_F$ and $K_F^*$ are connected for each finite subtree $F$ of $T$.
Recall that any quotient of a locally connected space is locally connected.
Since $K_F$ is obtained by gluing finitely many locally connected spaces, it is locally connected itself.
Since $K_F^*$ is a quotient of $K_F$, it is also locally connected.

In order to apply Theorem~\ref{thm:Capel} to see that $\lim\Theta$ is locally connected, it suffices to check that the bonding maps $K^*_{F_2} \to K^*_{F_1}$ are monotone.
Any nontrivial point preimage is a quotient of $K_{F}$ for some subtree $F$ of $F_2 - F_1$, which must be compact and connected.  Thus $\lim\Theta$ is locally connected.
\end{proof}

\section{Putting together the pieces}
\label{sec:PuttingTogether}

In this section, we complete the proof of Theorem~\ref{thm:MainThm}. Suppose $G$ is a one-ended group acting geometrically on a $\CAT(0)$ space $X$ with isolated flats, and assume that the maximal peripheral splitting $\mathcal{G}$ of $G$ is locally finite.
The results of this section lead up to Proposition~\ref{prop:TreeDecomposition}, which states that $\boundary X$ is homeomorphic to the limit of a tree system of spaces, whose underlying tree is the Bass--Serre tree $T$ for the splitting $\mathcal{G}$ and whose component spaces are the boundaries of the component vertex groups.
The proof of Theorem~\ref{thm:MainThm} will follow by combining Proposition~\ref{prop:TreeDecomposition} with ingredients established in the previous sections.

Let $X$ be the given $\CAT(0)$ space with isolated flats on which $G$ acts geometrically.
To simplify some of the geometric arguments, we will replace $X$ with a quasi-isometric space $X_T$ obtained by Bridson--Haefliger's Equivariant Gluing construction for graphs of groups with $\CAT(0)$ vertex groups and convex edge groups (see Theorems II.11.18 and~II.11.21 of \cite{BH99}). 

The space $X_T$ is constructed as follows.
Recall that the vertices of $T$ have two types: component vertices and peripheral vertices.
We denote the set of component vertices by $\mathcal C$ and the set of peripheral vertices by $\mathcal P$.  Each component vertex group $G_v$ acts geometrically on a $\CAT(0)$ space $C_v$.  Each peripheral vertex group $P_w$, being virtually abelian, acts geometrically on a flat Euclidean space $F_w$.

Each edge $e$ in $T$ is incident to a unique peripheral vertex $w \in \mathcal{P}$.
Let $F_e$ be equal to the flat $F_w$.
Each edge group $P_e$ also comes with a monomorphism $\phi_e \colon P_e \hookrightarrow G_v$.
By the Flat Torus Theorem there is a $\phi_e$--equivariant isometric embedding of the flat $F_e$ in the space $C_v$.

The space $X_T$ is obtained by gluing all components $C_v$ and flats $F_w$ using edge spaces of the form $F_e\times[0,1]$ in the pattern of the tree $T$.
For each edge $e$ incident to vertices $v\in\mathcal C$ and $w\in\mathcal P(T)$, the map from $F_e$ to $F_w$ is the identity while the map from $F_e$ to $C_v$ is the map given by the Flat Torus Theorem. 
Our setup is now a special case of the Equivariant Gluing discussed in \cite[Chapter II.11]{BH99}.  Therefore $G$ acts geometrically on $X_T$.

From this point on, we work in the space $X_T$ instead of $X$.  Since $\CAT(0)$ groups with isolated flats have ``well-defined boundaries,'' the spaces $X$ and $X_T$ have $G$--equivariantly homeomorphic boundaries (see  Theorem~\ref{thm:WellDefinedBoundary}).

Suppose $x_0\in X_T$ is a basepoint in a component space $C_0$.  Let $v_0 \in\mathcal{C}(T)$ be the vertex corresponding to $C_0$.  Let $\xi\in\boundary X_T$.  Using the terminology in \cite{CrokeKleiner00}, we can assign an \emph{itinerary} to $\xi$ at $x_0$.  This consists of a sequence of edges $\{e_i\}$ of $T$ corresponding to the sequence of edge spaces $F_i\times [0,1]$ that $\xi$ enters when based at $x_0$.  We say a ray \emph{enters} an edge space $F_i \times [0,1]$ if the ray reaches a point of the interior $F_i \times (0,1)$.   Observe that $\xi$ has an empty itinerary if and only if $\xi \in \boundary C_0$.
The next result is analogous to Lemma~2 in \cite{CrokeKleiner00}.

\begin{lem} If $\xi \notin \boundary C_0$, then the itinerary to $\xi$ at $x_0$ is the sequence of successive edges of a geodesic segment or geodesic ray beginning at $v_0$ in the tree $T$.
\end{lem}

\begin{proof}
The separation properties of edge spaces imply that successive edges in the itinerary must be adjacent in $T$, so the itinerary defines a path in $T$.  A geodesic that enters an edge space $F\times [0,1]$ through the flat $F\times \{0\}$ must exit through the flat $F \times \{1\}$ without backtracking.
Edge spaces are convex so a geodesic cannot revisit any edge space which it has left.
Therefore the corresponding path in $T$ is a geodesic.
\end{proof}

\begin{lem}\label{lem:idealpts}
Let $x_0$ be a basepoint contained in a component $C_{v_0}$.
   \begin{enumerate}
   \item
   \label{item:idealpts-surjective}
   If a geodesic ray based at $x_0$ has an infinite itinerary, then that itinerary is a geodesic ray of $T$ based at $v_0$.
   \item
   \label{item:idealpts-WellDef}
   Every ray of $T$ based at $v_0$ is the itinerary of a geodesic ray based at $x_0$.
   \item
   \label{item:idealpts-injective}
   If $c,c'$ are geodesic rays based at $x_0$ in $X_T$ that have the same infinite itinerary at $x_0$, then $c=c'$.
   \end{enumerate}
\end{lem}

\begin{proof}
In order to show (\ref{item:idealpts-surjective}),
it suffices to verify that the infinite itinerary is based at $v_0$, which is clear since the first flat the ray enters must be adjacent to the component $C_{v_0}$.

Any geodesic ray $\{e_i\}_{i=1}^\infty$ in $T$ based at $v_0$ corresponds to a sequence of edge spaces $F_i \times [0,1]$.
Any geodesic segment $c_i$ from $x_0$ to $F_i\times\{1/2\}$ must enter the edge space $F_j \times [0,1]$ for all $1\le j \le i$ since that edge space separates $x_0$ from $F_i \times \{1/2\}$.
After passing to a subsequence the geodesics $c_i$ converge to a ray $c$ based at $x_0$ which enters all edge spaces $F_i \times [0,1]$ for $i=1,2,3,\dots$ since the edge spaces separate $X_T$.  It follows that the given geodesic in $T$ is the itinerary of $c$ at $x_0$, establishing (\ref{item:idealpts-WellDef}).

Since the itinerary is infinite, there is a sequence of edge spaces $F_i \times [0,1]$ that $c,c'$ both enter. Let $F_i$ denote the flat $F_i \times \{1/2\}$ in the $i$th edge space. 
Using Theorem~\ref{thm:facts}, there exists a constant $\kappa$ such that for each $i$, if $q_i$ is the closest point of $F_i$ to $x_0$ then $[x_0,q_i]\cup F_i$ is $\kappa$--quasiconvex in $X$.  Thus both $c$ and $c'$ come within a distance $3\kappa$ of $q_i$.   In particular, $c,c'$ come $6\kappa$--close to each other arbitrarily far from $x_0$.  By convexity of the distance function in $X_T$, this implies $c=c'$, which proves (\ref{item:idealpts-injective}). 
\end{proof}

\begin{lem}\label{lem:dichotomy}If $\xi\in\boundary X_T$, then exactly one of the following holds:
\begin{enumerate}
\item $\xi\in\boundary C$ for some component $C$.  This includes the case that $\xi\in\boundary F$ for some flat.
\item $\xi$ has an infinite itinerary with respect to any $x_0 \in X_T$
and is not contained in $\boundary C$ for any component $C$ in $X_T$. 
\end{enumerate}
\end{lem} 

\begin{proof}  Let $c$ be a geodesic ray based at $x_0$ representing $\xi$.
Note that $\xi$ has a finite itinerary if and only if $c$ enters only finitely many edge spaces,
which holds if and only if $c$ eventually remains in a component $C_v$ for some $v \in \mathcal{C}$.  That component is $C_{v_0}$, if and only if the itinerary of $\xi$ at $x_0$ is empty as was previously noted. 
Otherwise, the itinerary of $\xi$ at $x_0$ is the geodesic segment $[v_0,v]$ in $T$, and $\xi \in \boundary C_v$.

The property of having an infinite itinerary does not depend on the choice of basepoint $x_0 \in X_T$.  Suppose $\xi$ has an infinite itinerary
and $C_v$ is any component.  Then $\xi \notin \boundary C_v$.
Indeed this becomes obvious if we choose a basepoint $x_0$ from the convex set $C_v$.
\end{proof}

\begin{lem}\label{lem:treelike}
Let $C_v$ and $C_{v'}$ be two distinct components of $X_T$.
Then one of the following holds:
\begin{enumerate}
\item $\boundary C_v\cap\boundary C_{v'}=\emptyset$.
\item $\boundary C_v\cap\boundary C_{v'}=\boundary F_w$ for some $w \in \mathcal{P}(T)$ adjacent to both $v$ and $v'$.
In this case, there is a copy of $F_w\times[-1,1]$ embedded in $X_T$ with $F_w\times\{-1\}\subset C_v$ and $F_w\times\{1\}\subset C_{v'}$.
\end{enumerate}
\end{lem}

\begin{proof}
Suppose $\xi \in \boundary C_{v} \cap \boundary C_{v'}$ for some $v \ne v'$.  Then $d_T (v,v')=2$.  Indeed if this distance were greater than two then $C_v$ and $C_{v'}$ would be separated by a pair of distinct flats $F_{e} \ne F_{e'}$ from the family of isolated flats $\mathcal{F}$. The existence of a ray asymptotic to both $C_v$ and $C_v'$ would then contradict isolated flats.

Let $w \in \mathcal{P}$ be the unique peripheral vertex adjacent to both $v$ and $v'$.
Note that the two edge spaces between $v$ and $v'$ are each isometric to $F_w \times [0,1]$, so their union is isometric to $F_w \times [-1,1]$.
Observe that this copy of $F_w \times [-1,1]$ separates $C_v$ from $C_{v'}$ in $X_T$.
Thus for any constant $r$, the intersection $N_r(C_v) \cap N_r(C_{v'})$ lies in a finite tubular neighborhood of $F_w$.
It follows that $\xi\in \boundary F_w$.
\end{proof}

The results above hold for any one-ended $\CAT(0)$ group with isolated flats.  For the rest of the section we are concerned only with the special case in which the maximal peripheral splitting is locally finite.

It is our goal to show that, in this case, $\boundary X_T$ is homeomorphic to the limit of the tree system $\Theta$ defined in the following construction.

\begin{construction}[The tree system of the peripheral splitting]
\label{const:TheTreeSystem}
Recall that the one-ended group $G$ has a maximal peripheral splitting $\mathcal{G}$.  By hypothesis, this splitting is assumed to be locally finite, which means that each peripheral vertex $w \in \mathcal{P}$ of the Bass--Serre tree $T$ has finite valence.
To each vertex $v \in \mathcal{C}$ we associate the subspace $K_v = \boundary C_v$, and to each $w \in \mathcal{P}$ we associate the subspace $K_w = \boundary F_w$.
To each oriented edge $e \in \mathcal{E}$ we associate the subspace $\Sigma_e=\boundary F_e$.  Since $\Sigma_e = \Sigma_{\bar{e}}$, we set $\phi_e$ to be the identity map.
Since $\Sigma_e \subseteq K_{t(e)}$ we set $i_e$ to be the inclusion.
For each $v \in \mathcal{C}$, the family of closed subspaces $\bigset{i_e(\Sigma_e)}{t(e)=v}$ is pairwise disjoint by the definition of isolated flats. Proposition~\ref{prop:nullcondition} will imply that this family is null.
Thus the data above define a tree system $\Theta$ whose underlying tree is $T$.
\end{construction}

The following proposition follows easily from the conclusions of Lemmas \ref{lem:idealpts}, \ref{lem:dichotomy}, and~\ref{lem:treelike} about itineraries and their relation to the structure of $\boundary X_T$.
As in Section~\ref{sec:TreeLimits}, given a tree system $\Theta$, we let $\#\Theta$ denote the quotient space obtained by gluing the vertex spaces of $\Theta$ along edge spaces via the maps $i_e$.

\begin{prop}
\label{prop:rhoProperties}
There is a map $\rho\colon \#\Theta \cup \boundary T \to \boundary X_T$ with the following properties:
  \begin{enumerate}
  \item
  \label{item:rhoBijection}
  $\rho$ is a bijection.
  \item $\rho$ is continuous on $\#\Theta$.
  \item For each finite subtree $F$ of $T$, the map $\rho$ restricted to the partial union $K_F$ is a topological embedding.  In particular, $\rho$ is an embedding when restricted to any vertex space $K_v$. \qed
  \end{enumerate}
\end{prop}

If $S$ is a subtree of $T$, we write $\Psi(S) = \rho \bigl( \#\Theta_S \cup \boundary S \bigr)$.  If $T_e$ is a branch of $T$, the set $\Psi(T_e)$ is a \emph{branch} of $\boundary X_T$.

\begin{prop}
\label{prop:nullcondition}
Choose a basepoint $v \in \mathcal{V}$.  The family of all branches $\Psi(T_e)$ of $\boundary X_T$ such that $e$ points away from $v$ is a null family.
\end{prop}

The proof of the previous proposition uses the following two lemmas.

\begin{lem}
\label{lem:FiniteItineraries}
Let $x_0$ be a basepoint contained in a component $C_{v_0}$.
   \begin{enumerate}
   \item
   \label{item:FI-comp}
   If $\xi \in \boundary C_v$ for some $v \in \mathcal{C}$ but not in $\boundary F_w$ for any $w$, then the itinerary of $\xi$ at $x_0$ is the geodesic segment $[v_0,v]$.
   \item
   \label{item:FI-flat}
   If $\xi \in \boundary F_w$ for some $w\in \mathcal P$, then the itinerary of $\xi$ at $x_0$ is the geodesic segment $[v_0,v]$ of $T$, where $v$ is the vertex adjacent to $w$ that is closest to $v_0$.
   \end{enumerate}
\end{lem}

\begin{proof}
In case (\ref{item:FI-comp}), it follows from Lemma~\ref{lem:treelike} that $C_v$ is the unique component whose boundary contains $\xi$.
The proof of Lemma~\ref{lem:dichotomy} implies that $\xi$ eventually remains in $C_v$, and the itinerary of $\xi$ at $x_0$ equals $[v_0,v]$ as desired.

If $\xi \in \boundary F_w$ then $\xi$ also lies in $\boundary C_v$ for all $v$ adjacent to $w$ in $T$. Furthermore, $\xi$ does not lie in $\boundary C'$ for any other component $C'$.
The component that $\xi$ eventually remains in must therefore be adjacent to $F_w$.
Observe that $\xi$ cannot enter any edge space $F_e \times [0,1]$ adjacent (and parallel) to $F_w$ since then $\xi$ would fail to be asymptotic to $F_w$.
\end{proof}

\begin{lem}
\label{lem:smallbranch}
Let $\Psi(T_e)$ be a branch of $\boundary X_T$ determined by an oriented edge $e \in \mathcal{E}$.
Let $x_0\in X_T$ be a basepoint contained in $C_v$ for some vertex $v \notin T_e$.
Let $\kappa$ be the constant from Theorem~\ref{thm:facts}.
If $d(x_0,F_e)\ge r + 3\kappa$,
then there exists a geodesic ray $c$ based at $x_0$
such that $\Psi(T_e) \subseteq U(c,r,7\kappa)$.
\end{lem}

\begin{proof} 
Let $q$ be the nearest point in $F_e$ to $x_0$.  Then $d(x_0,q) \ge r+3\kappa$.  Let $c,c'$ be geodesic rays based at $x_0$ that both intersect $F_e$. By Theorem~\ref{thm:facts} the set $[x_0,q] \cup F_e$ is $\kappa$--quasiconvex.  Thus there exist $s,s'$ with $c(s),c'(s')$ both contained in $\bignbd{[x_0,q]}{\kappa}\cap \nbd{F_e}{\kappa}$.  It follows that $d\bigl(c(s),q\bigr)$ and $d\bigl(c'(s'),q\bigr)$ are each less than $3\kappa$.  Thus $d\bigl(c(s),c'(s')\bigr) < 6\kappa$.  By the Law of Cosines $d\bigl(c(r),c'(r)\bigr)<6\kappa$.  In particular $c' \in U(c,r,6\kappa)$.

If $c$ intersects $F_e$ and $c'$ is asymptotic to $F_e$ (but does not intersect $F_e$), then $c'$ is a limit of geodesics that intersect $F_e$.  In this case, we conclude that $c' \in U(c,r,7\kappa)$.

By Lemma~\ref{lem:FiniteItineraries} each ray based at $x_0$ and asymptotic to $\Psi(T_e)$ has an itinerary involving $F_e$---and hence intersects $F_e$---unless the ray is asymptotic to $F_e$ itself.  In all cases we see that $\Psi(T_e) \subseteq U(c,r,7\kappa)$ for any $c$ crossing $F_e$.
\end{proof}

\begin{proof}[Proof of Proposition~\ref{prop:nullcondition}]  Let $v\in \mathcal C$.
Choose a basepoint $x_0 \in C_v$.
Let $D=7\kappa$ where $\kappa$ is as in the previous lemma and let $r<\infty$. Let $\Psi(T_e)$ be a branch such that $e$ points away from $v$ and such that $\Psi(T_e)$ is not contained in any set of the form $U(\cdot,r,D)$.  Applying Lemma~\ref{lem:smallbranch} in the contrapositive implies $d(x_0,F_e) < r+3\kappa$.  
Since the collection of flats $\{F_e\}$ in $X_T$  is locally finite, there are only finitely many edges $e$ whose corresponding flat $F_e$ is that close to $x_0$.  Thus there are only finitely many possibilities for the branch $\Psi(T_e)$.
\end{proof}

\begin{prop}
\label{prop:BranchesClosed}
Branches $\Psi(T_e)$ are closed in $\boundary X_T$.
\end{prop}

\begin{proof}
Let $x_0\in X_T$ be a basepoint contained in $C_v$ for some vertex $v\notin T_e$. 
Suppose $\{c_i\}$ is a sequence of geodesic rays in $X_T$ based at $x_0$ asymptotic to $\Psi(T_e)$ such that $c_i$ converges to the geodesic ray $c$ based at $x_0$.  As in the proof of Lemma~\ref{lem:smallbranch}, each such ray intersects $F_e$ or is asymptotic to $F_e$.  By passing to a subsequence we can assume all $c_i$ intersect $F_e$ or all $c_i$ are asymptotic to $F_e$.  

If each $c_i$ is asymptotic to $F_e$, then $c$ is as well, since $F_e$ is a closed convex subspace of $X_T$.
On the other hand, if each $c_i$ intersects $F_e$ then $c$ either intersects $F_e$ or is asymptotic to $F_e$ depending on whether the intersections of the $c_i$ with $F_e$ remain bounded as $i \to \infty$.  In all cases we conclude that $c$ is asymptotic to the branch $\Psi(T_e)$.
\end{proof}

Recall that $\lim \Theta$ is the inverse limit of an inverse system of spaces $K_F^* = K_F/\mathcal{A}_F$ for all finite subtrees $F$ of $T$ (see Definition~\ref{defn:TreeLimit}).  The collection $\mathcal{A}_F$ contains an edge space $i_e(\Sigma_e)$ for each edge $e \in N_F$, where $N_F$ contains all edges whose origin is outside $F$ and whose terminus is in $F$.

\begin{lem}
\label{lem:SameQuotient}
For each finite subtree $F$ of $T$, let $\boundary X_T / \mathcal{D}_F$ be the quotient of $\boundary X_T$ formed by collapsing each branch $\Psi(T_e)$ to a point whenever $\bar{e} \in N_F$.
Let $q_F\colon\boundary X\to \boundary X/\mathcal{D}_F$ denote the natural quotient map.  Then $\boundary X_T/\mathcal{D}_F$ is homeomorphic to the quotient $K^*_F = K_F/\mathcal{A}_F$. 
\end{lem}

\begin{proof}
The embedding $\rho \colon K_F \to \boundary X_T$ induces a continuous map $g_F \colon K_F^* \to \boundary X_T / \mathcal{D}_F$, which is clearly a bijection.  It suffices to verify that $g_F$ is a closed map.
We will show that the decomposition $\mathcal{D}_F$ is upper semicontinuous.  It then follows immediately from Proposition~\ref{prop:Metrizable} that $\boundary X_T / \mathcal{D}_F$ is Hausdorff, and since $K_F^*$ is compact, we can conclude that $g_F$ is closed. 

By Proposition~\ref{prop:rhoProperties}(\ref{item:rhoBijection}), two branches $\Psi(T_e)$ and $\Psi(T_{e'})$ with $\bar{e},\bar{e}' \in N_F$ intersect precisely when their origin vertices $o(e)$ and $o(e')$ are peripheral vertices (lying in $\mathcal{P}$) and are equal.  In this case the intersection of the branches is the peripheral vertex space $\Sigma_e = K_{o(e)} = \Sigma_{e'}$.  Therefore each nontrivial member of the decomposition $\mathcal{D}_F$ is the union of finitely many branches whose defining edges have a common origin.

By Proposition~\ref{prop:nullcondition} the branches being collapsed are a null family.  Therefore $\mathcal{D}_F$ is also null.  Similarly by Proposition~\ref{prop:BranchesClosed} we see that the members of $\mathcal{D}_F$ are closed.  Therefore $\mathcal{D}_F$ is upper semicontinuous by Proposition~\ref{prop:NullUSC}.
\end{proof}

\begin{prop}
\label{prop:TreeDecomposition}
Suppose $G$ is a one-ended group acting geometrically on a $\CAT(0)$ space $X$ with isolated flats.
Suppose the maximal peripheral splitting of $G$ is locally finite.
Then the boundary $\boundary X$ is homeomorphic to the limit $\lim\Theta$ of the associated tree system.
\end{prop}

\begin{proof}
By Theorem~\ref{thm:WellDefinedBoundary}, the boundary $\boundary X$ is $G$--equivariantly homeomorphic to the boundary $\boundary X_T$ constructed above.
In order to prove that $\boundary X_T$ is homeomorphic to $\lim\Theta$ we will define maps $h_F$ from $\boundary X_T$ onto each $K_F^*$ in the inverse system, and show that the induced map $h\colon \boundary X_T \to \lim\Theta$ is a homeomorphism.

The map $h_F$ is the composition $g_F^{-1}\of q_F$, where $g_F$ is the homeomorphism defined in Lemma~\ref{lem:SameQuotient} and $q_F$ is the natural quotient map $\boundary X_T \to \boundary X_T / \mathcal{D}_F$.
Since $f_{F_1 F_2} \of h_{F_2} = h_{F_1}$ whenever $F_1 \subseteq F_2$, the maps $h_F$ induce a continuous map $h \colon \boundary X_T \to \lim \Theta$.

Observe that $\boundary X_T$ is compact, each quotient space $K_F^*$ is Hausdorff, and $h_F$ is surjective for each finite subtree $F$.  It follows that $h$ is surjective (see for instance \S I.9.6, Corollary~2(b) of \cite{BourbakiTG}).  Since $\lim\Theta$ is Hausdorff, the map $h$ is closed.

Thus we only need to show that $h$ is injective.
Suppose $\xi \ne \eta$ are two distinct points of $\boundary X_T$.
Recall that by Proposition~\ref{prop:rhoProperties}(\ref{item:rhoBijection}) each point of $\boundary X_T$ is either contained in a block boundary $K_v$ for some $v \in \mathcal{C}$ or is equal to an ideal point $\rho(z)$ for some $z \in \boundary T$.  We also know that the map $\rho\colon \boundary T \to \boundary X_T$ from ends of the tree to ideal points is injective.

Case~1: Suppose $\xi$ and $\eta$ are both contained in block boundaries.
Recall that each point of $\boundary X_T$ lies in at most finitely many block boundaries.
Choose a finite subtree $F$ that contains all block boundaries $K_v$ that contain either $\xi$ or $\eta$.
Then $\xi$ and $\eta$ are contained in the partial union $K_F$, but neither $\xi$ nor $\eta$ is contained in any subspace $i_e(\Sigma_e)$ with ${e} \in N_F$.
It follows that $\xi$ and $\eta$ have distinct images $h_F(\xi)$ and $h_F(\eta)$ in the quotient $K_F^*$.

Case~2: Suppose $\xi$ and $\eta$ are equal to distinct ideal points $\rho(z)$ and $\rho(z')$ with $z\ne z'$ in $\boundary T$.  Let $c$ be the geodesic in the tree $T$ from $z$ to $z'$.  Choose any vertex $v \in \mathcal{C}$ of $c$, and let $F$ be the finite subtree $\{v\}$. Then $K_F=K_v$ consists of only one block boundary.  Since $c$ is a geodesic, there are distinct edges $e\ne e'$ with $o(e)=o(e')=v$ such that $\xi$ lies in the branch $\Psi(T_e)$ and $\eta$ lies in the branch $\Psi(T_{e'})$.
Therefore $\xi$ and $\eta$ have distinct images $h_F(\xi)$ and $h_F(\eta)$ in the quotient $K_F^*$.

Case~3: Suppose $\xi$ is contained in a block boundary and $\eta$ is an ideal point $\rho(z)$ for $z \in \boundary T$.
Choose a finite subtree $F$ containing all of the finitely many blocks $K_v$ that contain $\xi$.
Then $\xi$ lies in the partial union $K_F$ but is not contained in any branch $\Psi(T_e)$ with $\bar{e} \in N_F$.
On the other hand, $\eta$ is not contained in $K_F$, so it must lie in some branch $\Psi(T_e)$ with $\bar{e} \in N_F$.  It follows that $\xi$ and $\eta$ have distinct images $q_F(\xi)$ and $q_F(\eta)$ in the quotient $\boundary X_T / \mathcal{D}_F$.  Consequently their images in $K_F^*$ are distinct as well.
\end{proof}

We now use Proposition~\ref{prop:TreeDecomposition} to complete the proof of Theorem~\ref{thm:MainThm}.

\begin{proof}[Proof of Theorem~\ref{thm:MainThm}]
The reverse implication follows from work of Mihalik--Ruane \cite{MihalikRuane99,MihalikRuane01}.

We prove the forward implication using the contrapositive; i.e., if $G$ does not split over a virtually abelian subgroup as in the statement of Theorem~\ref{thm:MainThm}, we must show that $\boundary X$ is locally connected.

Let $\mathcal{G}$ be the maximal peripheral splitting of $G$ given by Theorem~\ref{thm:Bowditch}.  By hypothesis, this splitting is locally finite.
Proposition~\ref{prop:TreeDecomposition} implies that $\boundary X$ is homeomorphic to the limit of the associated tree system.
The trivial case in which $G$ is virtually abelian of higher rank is obvious since the boundary is a sphere in that case.
In all other cases the tree system is nondegenerate.
By Corollary~\ref{cor:ComponentCAT0} each component vertex group is an atomic $\CAT(0)$ group with isolated flats.
By Theorem~\ref{thm:UnsplittableLC} each component vertex space is connected and locally connected.
Each peripheral vertex space is a sphere of dimension at least one, hence is connected.
Therefore we may apply Theorem~\ref{thm:TreeLimLC} to conclude that $\boundary X$ is locally connected, as desired.
\end{proof}

\bibliographystyle{alpha}
\bibliography{chruska.bib}

\def\polhk#1{\setbox0=\hbox{#1}{\ooalign{\hidewidth
  \lower1ex\hbox{$\,\lhook$}\hidewidth\crcr\unhbox0}}}
  \def\RomanianComma#1{\setbox0=\hbox{#1}{\ooalign{\hidewidth
  \lower1.2ex\hbox{$\mspace{1mu}^{,}$}\hidewidth\crcr\unhbox0}}}
\begin{thebibliography}{DGP11}

\bibitem[AS85]{AncelSiebenmann85}
F.D. Ancel and L.~Siebenmann.
\newblock The construction of homogeneous homology manifolds.
\newblock {\em Abstracts Amer. Math. Soc.}, 6, 1985.
\newblock Abstract number 816-57-72.

\bibitem[Bal95]{Ballmann95}
Werner Ballmann.
\newblock {\em Lectures on spaces of nonpositive curvature}, volume~25 of {\em
  DMV Seminar}.
\newblock Birkh\"auser Verlag, Basel, 1995.

\bibitem[Ben92]{Benakli92_thesis}
N.~Benakli.
\newblock {\em Poly\`{e}dres hyperboliques: passage du local au global}.
\newblock PhD thesis, Universit\'{e} Paris--Sud, 1992.

\bibitem[Bes96]{Bestvina96}
Mladen Bestvina.
\newblock Local homology properties of boundaries of groups.
\newblock {\em Michigan Math. J.}, 43(1):123--139, 1996.

\bibitem[BH99]{BH99}
Martin~R. Bridson and Andr{\'e} Haefliger.
\newblock {\em Metric spaces of non-positive curvature}.
\newblock Springer-Verlag, Berlin, 1999.

\bibitem[BM91]{BestvinaMess91}
Mladen Bestvina and Geoffrey Mess.
\newblock The boundary of negatively curved groups.
\newblock {\em J. Amer.\ Math.\ Soc.}, 4(3):469--481, 1991.

\bibitem[Bou71]{BourbakiTG}
N.~Bourbaki.
\newblock {\em \'{E}l\'ements de math\'ematique. {T}opologie g\'en\'erale.
  {C}hapitres 1 \`a 4}.
\newblock Hermann, Paris, 1971.

\bibitem[Bou97]{Bourdon97}
M.~Bourdon.
\newblock Immeubles hyperboliques, dimension conforme et rigidit\'e de
  {M}ostow.
\newblock {\em Geom. Funct. Anal.}, 7(2):245--268, 1997.

\bibitem[Bow98]{Bowditch98CutPoints}
Brian~H. Bowditch.
\newblock Cut points and canonical splittings of hyperbolic groups.
\newblock {\em Acta Math.}, 180(2):145--186, 1998.

\bibitem[Bow99]{Bowditch99Treelike}
B.H. Bowditch.
\newblock Treelike structures arising from continua and convergence groups.
\newblock {\em Mem.\ Amer.\ Math.\ Soc.}, 139(662):1--86, 1999.

\bibitem[Bow01]{Bowditch01}
B.H. Bowditch.
\newblock Peripheral splittings of groups.
\newblock {\em Trans. Amer. Math. Soc.}, 353(10):4057--4082, 2001.

\bibitem[Bow12]{BowditchRelHyp}
B.H. Bowditch.
\newblock Relatively hyperbolic groups.
\newblock {\em Internat. J. Algebra Comput.}, 22(3):1250016, 66 pages, 2012.

\bibitem[BZ]{Benzvi_PathCon}
Michael Ben-Zvi.
\newblock Path connectedness of $\text{CAT}(0)$ groups with the isolated flats
  property.
\newblock In preparation.

\bibitem[Cap54]{Capel54}
C.~E. Capel.
\newblock Inverse limit spaces.
\newblock {\em Duke Math. J.}, 21:233--245, 1954.

\bibitem[Cha95]{Champetier95}
Christophe Champetier.
\newblock Propri\'{e}t\'{e}s statistiques des groupes de pr\'{e}sentation
  finie.
\newblock {\em Adv. Math.}, 116(2):197--262, 1995.

\bibitem[CK00]{CrokeKleiner00}
Christopher~B. Croke and Bruce Kleiner.
\newblock Spaces with nonpositive curvature and their ideal boundaries.
\newblock {\em Topology}, 39(3):549--556, 2000.

\bibitem[Dav86]{Daverman86}
Robert~J. Daverman.
\newblock {\em Decompositions of manifolds}, volume 124 of {\em Pure and
  Applied Mathematics}.
\newblock Academic Press, Inc., Orlando, FL, 1986.

\bibitem[DGP11]{DahmaniGuirardelPrzytycki11}
Fran\c{c}ois Dahmani, Vincent Guirardel, and Piotr Przytycki.
\newblock Random groups do not split.
\newblock {\em Math. Ann.}, 349(3):657--673, 2011.

\bibitem[DHW]{DaniHaulmarkWalsh_Nonplanar}
Pallavi Dani, Matthew Haulmark, and Genevieve Walsh.
\newblock Right-angled {C}oxeter groups with non-planar boundary.
\newblock Preprint. arXiv:1902.01029 [math.GT].

\bibitem[GS]{GeogheganSwenson_Semistable}
Ross Geoghegan and Eric Swenson.
\newblock On semistability of {$\CAT(0)$} groups.
\newblock Preprint. arXiv:1707.07061.

\bibitem[Gui14]{Guilbault14}
Craig~R. Guilbault.
\newblock Weak {Z}-structures for some classes of groups.
\newblock {\em Algebr. Geom. Topol.}, 14(2):1123--1152, 2014.

\bibitem[Han51]{Hanner51}
Olof Hanner.
\newblock Some theorems on absolute neighborhood retracts.
\newblock {\em Ark. Mat.}, 1:389--408, 1951.

\bibitem[Hau18]{HaulmarkCAT0}
Matthew Haulmark.
\newblock Boundary classification and two-ended splittings of groups with
  isolated flats.
\newblock {\em J. Topol.}, 11(3):645--665, 2018.

\bibitem[HHS]{HaulmarkHruskaSathaye_Menger}
Matthew Haulmark, G.~Christopher Hruska, and Bakul Sathaye.
\newblock Nonhyperbolic {C}oxeter groups with {M}enger curve boundary.
\newblock To appear in \textit{Enseign.\ Math.}. arXiv:1812.04649 [math.GR].

\bibitem[HK05]{HruskaKleinerIsolated}
G.~Christopher Hruska and B.~Kleiner.
\newblock Hadamard spaces with isolated flats.
\newblock {\em Geom. Topol.}, 9:1501--1538, 2005.
\newblock With an appendix jointly written by the authors and M.~Hindawi.

\bibitem[HK09]{HruskaKleinerErratum}
G.~Christopher Hruska and Bruce Kleiner.
\newblock Erratum to: ``{H}adamard spaces with isolated flats''.
\newblock {\em Geom. Topol.}, 13(2):699--707, 2009.

\bibitem[HR]{HruskaRuaneHierarchies}
G.~Christopher Hruska and Kim Ruane.
\newblock Hierarchies and semistability of relatively hyperbolic groups.
\newblock Preprint. arXiv:1904.12947 [math.GR].

\bibitem[Hru04]{Hruska2ComplexIFP}
G.~Christopher Hruska.
\newblock Nonpositively curved 2--complexes with isolated flats.
\newblock {\em Geom. Topol.}, 8:205--275, 2004.

\bibitem[Hu65]{Hu65}
Sze-Tsen Hu.
\newblock {\em Theory of retracts}.
\newblock Wayne State University Press, Detroit, 1965.

\bibitem[Jak80]{Jakobsche80}
W.~Jakobsche.
\newblock The {B}ing--{B}orsuk conjecture is stronger than the {P}oincar\'e
  conjecture.
\newblock {\em Fund. Math.}, 106(2):127--134, 1980.

\bibitem[Jak91]{Jakobsche91}
W.~Jakobsche.
\newblock Homogeneous cohomology manifolds which are inverse limits.
\newblock {\em Fund. Math.}, 137(2):81--95, 1991.

\bibitem[KK00]{KapovichKleiner00}
Michael Kapovich and Bruce Kleiner.
\newblock Hyperbolic groups with low-dimensional boundary.
\newblock {\em Ann. Sci. \'Ecole Norm. Sup. (4)}, 33(5):647--669, 2000.

\bibitem[KL95]{KapovichLeeb95}
M.~Kapovich and B.~Leeb.
\newblock On asymptotic cones and quasi-isometry classes of fundamental groups
  of $3$--manifolds.
\newblock {\em Geom.\ Funct.\ Anal.}, 5(3):582--603, 1995.

\bibitem[Lev98]{Levitt98}
Gilbert Levitt.
\newblock Non-nesting actions on real trees.
\newblock {\em Bull.\ London Math. Soc.}, 30(1):46--54, 1998.

\bibitem[LT17]{LouderTouikan17}
Larsen Louder and Nicholas Touikan.
\newblock Strong accessibility for finitely presented groups.
\newblock {\em Geom. Topol.}, 21(3):1805--1835, 2017.

\bibitem[Mih83]{Mihalik83}
Michael~L. Mihalik.
\newblock Semistability at the end of a group extension.
\newblock {\em Trans. Amer. Math. Soc.}, 277(1):307--321, 1983.

\bibitem[Mor16]{Moran16}
Molly~A. Moran.
\newblock Metrics on visual boundaries of {$\CAT(0)$} spaces.
\newblock {\em Geom. Dedicata}, 183:123--142, 2016.

\bibitem[MR99]{MihalikRuane99}
Michael Mihalik and Kim Ruane.
\newblock {$\CAT(0)$} groups with non--locally connected boundary.
\newblock {\em J. London Math. Soc. \textup{(}2\textup{)}}, 60(3):757--770,
  1999.

\bibitem[MR01]{MihalikRuane01}
Michael Mihalik and Kim Ruane.
\newblock {$\CAT(0)$} {HNN}-extensions with non--locally connected boundary.
\newblock {\em Topology Appl.}, 110(1):83--98, 2001.
\newblock Geometric topology and geometric group theory (Milwaukee, WI, 1997).

\bibitem[MS]{MihalikSwenson}
M.~Mihalik and E.~Swenson.
\newblock Relatively hyperbolic groups have semistable fundamental group at
  infinity.
\newblock Preprint, arXiv:1709.02420 [math.GR].

\bibitem[Ont05]{Ontaneda05}
Pedro Ontaneda.
\newblock Cocompact {$\CAT(0)$} spaces are almost geodesically complete.
\newblock {\em Topology}, 44:47--62, 2005.

\bibitem[OS15]{OsajdaSwiatkowski15}
Damian Osajda and Jacek \Swiatkowski.
\newblock On asymptotically hereditarily aspherical groups.
\newblock {\em Proc. Lond. Math. Soc. (3)}, 111(1):93--126, 2015.

\bibitem[Pon30]{Pontryagin30}
L.S. Pontryagin.
\newblock Sur une hypoth\`{e}se fundamentale de la th\'{e}orie de la dimension.
\newblock {\em C.R. Acad. Sci. Paris}, 190:1105--1107, 1930.

\bibitem[Rua05]{Ruane05Sierpinski}
Kim Ruane.
\newblock \textup{CAT}(0) boundaries of truncated hyperbolic space.
\newblock {\em Topology Proc.}, 29(1):317--331, 2005.

\bibitem[Ser77]{Serre77}
Jean-Pierre Serre.
\newblock {\em Arbres, amalgames, {${\rm SL}\sb{2}$}}, volume~46 of {\em
  Ast\'erisque}.
\newblock Soci\'et\'e Math\'ematique de France, Paris, 1977.
\newblock Written in collaboration with Hyman Bass.

\bibitem[Swa96]{Swarup96}
G.A. Swarup.
\newblock On the cut point conjecture.
\newblock {\em Electron.\ Res.\ Announc.\ Amer.\ Math.\ Soc.}, 2(2):98--100,
  1996.

\bibitem[{\'{S}}wi]{SwiatkowskiTreesOfCompacta}
J.~{\'{S}}wi{\k{a}}tkowski.
\newblock Trees of metric compacta and trees of manifolds.
\newblock Preprint. arXiv:1304.5064.

\bibitem[Tra13]{Tran13}
Hung~Cong Tran.
\newblock Relations between various boundaries of relatively hyperbolic groups.
\newblock {\em Internat. J. Algebra Comput.}, 23(7):1551--1572, 2013.

\bibitem[Wil70]{Willard70}
Stephen Willard.
\newblock {\em General topology}.
\newblock Addison-Wesley Publishing Co., Reading, Mass.-London-Don Mills, Ont.,
  1970.

\bibitem[Yam04]{Yaman04}
Asl{\i} Yaman.
\newblock A topological characterisation of relatively hyperbolic groups.
\newblock {\em J. Reine Angew. Math.}, 566:41--89, 2004.

\end{thebibliography}

\end{document}